\documentclass[oneside,english]{amsart}
\usepackage[T1]{fontenc}
\usepackage[latin9]{inputenc}
\usepackage{color}
\usepackage{float}
\usepackage{amstext}
\usepackage{amsthm}
\usepackage{amssymb}

\makeatletter

\providecommand{\tabularnewline}{\\}

\numberwithin{equation}{section}
\numberwithin{figure}{section}
\theoremstyle{plain}
\newtheorem{thm}{\protect\theoremname}
\theoremstyle{plain}
\newtheorem{lem}[thm]{\protect\lemmaname}
\theoremstyle{remark}
\newtheorem{rem}[thm]{\protect\remarkname}
\theoremstyle{definition}
\newtheorem{example}[thm]{\protect\examplename}
\theoremstyle{plain}
\newtheorem{prop}[thm]{\protect\propositionname}
\theoremstyle{plain}
\newtheorem{cor}[thm]{\protect\corollaryname}
\theoremstyle{remark}
\newtheorem{notation}[thm]{\protect\notationname}
\theoremstyle{remark}
\newtheorem*{acknowledgement*}{\protect\acknowledgementname}

\usepackage{pstricks,pst-node,pst-tree}
\usepackage{xy}
\usepackage{amsmath, amscd, amsthm, amssymb, graphics, xypic, mathrsfs, setspace, fancyhdr, bm, pdfsync, mathptmx}
\usepackage[a4paper,top=1.05in, bottom=1.05in, left=0.85in, right=0.85in]{geometry}
\usepackage[colorlinks=true,citecolor=teal,linkcolor=blue, pagebackref=false]{hyperref}

\makeatother

\usepackage{babel}
\providecommand{\acknowledgementname}{Acknowledgement}
\providecommand{\corollaryname}{Corollary}
\providecommand{\examplename}{Example}
\providecommand{\lemmaname}{Lemma}
\providecommand{\notationname}{Notation}
\providecommand{\propositionname}{Proposition}
\providecommand{\remarkname}{Remark}
\providecommand{\theoremname}{Theorem}

\begin{document}
\title{Fibrations by affine lines on rational affine surfaces with irreducible
boundaries }
\author{Adrien Dubouloz}
\address{IMB UMR5584, CNRS, Univ. Bourgogne Franche-Comté, F-21000 Dijon, France.}
\email{adrien.dubouloz@u-bourgogne.fr}
\thanks{This research was partially funded by ANR Grant \textquotedblright FIBALGA\textquotedblright{}
ANR-18-CE40-0003. The Institute of Mathematics of Burgundy receives
support from the EIPHI Graduate School ANR-17- EURE-0002}
\begin{abstract}
We consider fibrations by affine lines on smooth affine surfaces obtained
as complements of smooth rational curves $B$ in smooth projective
surfaces $X$ defined over an algebraically closed field of characteristic
zero. We observe that except for two exceptions, these surfaces $X\setminus B$
admit infinitely many families of $\mathbb{A}^{1}$-fibrations over
the projective line with irreducible fibers and a unique singular
fiber of arbitrarily large multiplicity. For $\mathbb{A}^{1}$-fibrations
over the affine line, we give a new and essentially self-contained
proof that the set of equivalence classes of such fibrations up to
composition by automorphisms at the source and target is finite if
and only if the self-intersection number $B^{2}$ of $B$ in $X$
is less than or equal to $6$. 
\end{abstract}

\maketitle

\section*{Introduction}

Affine surfaces whose automorphism groups act with a dense orbit with
finite complement were studied by M.H. Gizatullin and V.I. Danilov
in a series of seminal papers in the seventies \cite{Giz70,Giz71, GiDa75, GiDa77}.
There, they established that except for finitely many special cases,
these are the affine surfaces which admit projective completions whose
boundaries are chains of smooth proper rational curves. Up to the
exception $\mathbb{A}^{1}\setminus\{0\}\times\mathbb{A}^{1}$, such
surfaces are equivalently characterized by the property that they
admit two fibrations over the affine line $\mathbb{A}^{1}$, whose
general fibers are pairwise distinct and isomorphic to $\mathbb{A}^{1}$,
see \cite{Giz71,Be83,Dub05}. Many of these surfaces actually admit
infinitely many such $\mathbb{A}^{1}$-fibrations with pairwise different
general fibers up to the equivalence relation given by composition
by automorphisms on the source and the target. This richness of $\mathbb{A}^{1}$-fibrations
contributes in a central way to the complexity of their automorphism
groups, see e.g. \cite{BlanDub11,FKZ11, BlanDub15,DubLam15,Kov15,KPZ17}. 

In this article, we consider $\mathbb{A}^{1}$-fibrations on the subclass
consisting of affine surfaces $S$ defined over an algebraically closed
field $k$ of characteristic zero and which admit smooth projective
completions $X$ with boundary $B=X\setminus S$ isomorphic to the
projective line $\mathbb{P}^{1}$. A surface of this type is isomorphic
either to the affine plane $\mathbb{A}^{2}$, or to the complement
of a smooth conic in $\mathbb{P}^{2}$ or to the complement of an
ample section of a $\mathbb{P}^{1}$-bundle $\pi_{n}:\mathbb{F}_{n}=\mathbb{P}(\mathcal{O}_{\mathbb{P}^{1}}\oplus\mathcal{O}_{\mathbb{P}^{1}}(-n))\to\mathbb{P}^{1}$
for some $n\geq0$. Furthermore, the famous Danilov-Gizatullin isomorphism
theorem asserts that the isomorphism class of an affine surface of
the form $\mathbb{F}_{n}\setminus B$ depends only on the self-intersection
$B^{2}$ of the boundary $B$. 

While $\mathbb{A}^{2}$ and the complements of smooth conics in $\mathbb{P}^{2}$
only admit $\mathbb{A}^{1}$-fibrations over $\mathbb{A}^{1}$, a
surface $\mathbb{F}_{n}\setminus B$ admits an $\mathbb{A}^{1}$-fibration
$\pi:\mathbb{F}_{n}\setminus B\to\mathbb{P}^{1}$ given by the restriction
of the $\mathbb{P}^{1}$-bundle $\pi_{n}:\mathbb{F}_{n}\to\mathbb{P}^{1}$.
These fibrations are actually locally trivial $\mathbb{A}^{1}$-bundles,
and one can check that their equivalence classes are in one-to-one
correspondence with the orbits of the natural action of the group
$\mathrm{PGL_{2}(k)}$ on the space $\mathbb{P}(H^{1}(\mathbb{P}^{1},\mathcal{O}_{\mathbb{P}^{1}}(-B^{2}))$
(see Remark \ref{rem:Torsor}). In particular, for $B^{2}\geq3$,
these surfaces admit infinitely many equivalence classes of $\mathbb{A}^{1}$-fibrations
over $\mathbb{P}^{1}$. The geometry and equivalence classes of other
families of $\mathbb{A}^{1}$-fibrations over $\mathbb{P}^{1}$ on
surfaces $\mathbb{F}_{n}\setminus B$ has been much less studied than
those of $\mathbb{A}^{1}$-fibrations over $\mathbb{A}^{1}$. Our
first result, inspired by a construction due to Blanc-van Santen \cite{BlavS19}
of infinite families of pairwise non-equivalent closed embedding of
the affine line in the complement of the diagonal in $\mathbb{P}^{1}\times\mathbb{P}^{1}$,
reads as follows: 
\begin{thm}
\label{thm:A1Fib-over-P1}Let $(\mathbb{F}_{n},B)$ be a pair consisting
of a Hirzebruch surface $\pi_{n}:\mathbb{F}_{n}\to\mathbb{P}^{1}$
and an ample section $B$ of $\pi_{n}$. Then for every $m\geq4$,
there exist infinite families of equivalence classes of $\mathbb{A}^{1}$-fibrations
$\pi:\mathbb{F}_{n}\setminus B\to\mathbb{P}^{1}$ with a unique singular
fiber, irreducible and of multiplicity $m$.
\end{thm}

The proof of Theorem \ref{thm:A1Fib-over-P1} given in Section \ref{sec:A1-Fib-P1}
actually provides a natural bijection between a set of equivalence
classes of $\mathbb{A}^{1}$-fibrations $\pi:\mathbb{F}_{n}\setminus B\to\mathbb{P}^{1}$
with a unique singular fiber, irreducible and of multiplicity $m$,
and the elements of the set-theoretic quotient of the set of closed
points of a certain Zariski dense open subset of $\mathbb{P}^{m-1}$
by the algebraic action of a linear algebraic group whose general
orbits are at most $2$-dimensional. This construction strongly suggests
that by replacing the consideration of set-theoretic quotients by,
for instance, that of GIT quotients of suitable open subsets, one
should be able to strengthen Theorem \ref{thm:A1Fib-over-P1} in a
form asserting the existence of algebraic families of $\mathbb{A}^{1}$-fibrations
s $\pi:\mathbb{F}_{n}\setminus B\to\mathbb{P}^{1}$ with a unique
singular fiber, irreducible and of multiplicity $m$ parametrized
by the closed points of an algebraic variety of dimension $m-3$.
Tackling the necessary additional constructions which are needed to
give a rigorous and accurate formulation of this coarse moduli viewpoint
falls beyond the scope of the present article. \textcolor{blue}{}\\

In a second step, we consider equivalence classes of $\mathbb{A}^{1}$-fibrations
over $\mathbb{A}^{1}$ on affine surfaces $X\setminus B$. It is a
well-known result of Danilov-Gizatullin \cite{GiDa75,GiDa77} that
every such surface other than the complement of a smooth conic in
$\mathbb{P}^{2}$ has a unique equivalence class of smooth $\mathbb{A}^{1}$-fibration
over $\mathbb{A}^{1}$. As already observed by Danilov-Gizatullin
again, for every $d=1,\ldots,5$, the finiteness of the number of
equivalences classes of $\mathbb{A}^{1}$-fibrations over $\mathbb{A}^{1}$
on surfaces $X\setminus B$ with $B^{2}=d$ is then a consequence
of the finiteness of isomorphism types of pairs $(X,B)$ with $B^{2}=d$.
Over the field of complex numbers, equivalence classes of non-smooth
$\mathbb{A}^{1}$-fibrations over $\mathbb{A}^{1}$, equivalently
$\mathbb{A}^{1}$-fibrations having non-reduced components in their
degenerate fibers, have been extensively studied in \cite{FKZ11}
in a broader context, see especially Corollary 6.3.19 and Corollary
6.3.20 in \emph{loc. cit.}. The techniques there consist in first
constructing a finite-to-one correspondence between equivalence classes
of such $\mathbb{A}^{1}$-fibrations and collections of points in
a configuration space. The latter encodes the standard construction
of a completion of an $\mathbb{A}^{1}$-fibered smooth affine surface
$\pi:S\to\mathbb{A}^{1}$ into a $\mathbb{P}^{1}$-fibered smooth
projective surface $\bar{\pi}:\bar{S}\to\mathbb{P}^{1}$ obtained
by performing a suitable sequence of blow-ups of closed points starting
from a Hirzerbruch surface $\pi_{n}:\mathbb{F}_{n}\to\mathbb{P}^{1}$.
The second step consists in describing the possible configurations
and determining their respective numbers of moduli. A consequence
of this extensive description is that for every $B^{2}\geq7$, the
surface $X\setminus B$ admits families of pairwise non-equivalent
$\mathbb{A}^{1}$-fibrations $X\setminus B\to\mathbb{A}^{1}$ parametrized
by a space whose dimension is an increasing function of $B^{2}$ (see
\cite[Example 6.3.2]{FKZ11}). Our second result consists of an alternative
direct proof of the following theorem, based on the use of a different
point of view. 
\begin{thm}
\label{thm:MainThmA1Fib-Aff}Let $(X,B)$ be a pair consisting of
a smooth projective surface $X$ and an ample smooth rational curve
$B$ on $X$. Then the following alternative holds:

a) If $B^{2}\leq6$ then $X\setminus B$ admits at most seven equivalence
classes of $\mathbb{A}^{1}$-fibrations over $\mathbb{A}^{1}$,

b) If $B^{2}\geq7$ then the set of equivalence classes of $\mathbb{A}^{1}$-fibrations
$X\setminus B\to\mathbb{A}^{1}$ is infinite, of cardinality larger
than or equal to that of the field $k$.
\end{thm}

For $B^{2}\leq6$, the different equivalence classes are derived by
an explicit case by case study. For $B^{2}\geq7$, our argument is
based on the study of the equivalence classes of a subclass of $\mathbb{A}^{1}$-fibrations
$\mathbb{F}_{n}\setminus B\to\mathbb{A}^{1}$ which have an irreducible
component of multiplicity two inside their degenerate fiber. We show
that for every $B^{2}\geq7$, the set of equivalence classes of $\mathbb{A}^{1}$-fibration
of this type is infinite. More precisely, we actually construct explicit
families of equivalences classes of $\mathbb{A}^{1}$-fibered smooth
affine surfaces $S\to\mathbb{A}^{1}$ with a unique degenerate fiber,
irreducible and of multiplicity two, depending algebraically on a
parameter varying in an affine space of dimension $\left\lfloor \tfrac{B^{2}-5}{2}\right\rfloor $
(see Example \ref{exa:Moduli_Sl,s}) and which are all realized as
restrictions of $\mathbb{A}^{1}$-fibrations $\mathbb{F}_{n}\setminus B\to\mathbb{A}^{1}$
on suitable Zariski open subsets. This indicates in an indirect fashion
that the ``number of moduli'' of $\mathbb{A}^{1}$-fibrations over
$\mathbb{A}^{1}$ on surfaces $X\setminus B$ with $B^{2}=d$ is bounded
from below by $\left\lfloor \tfrac{d-5}{2}\right\rfloor $. \textcolor{blue}{}\\

The article is organized as follows. In section one, after setting
some notations, we review basic properties of smooth affine surfaces
completable by smooth rational curves. We then proceed in section
two to the proof of Theorem \ref{thm:A1Fib-over-P1}. The third section
is devoted to the proof of Theorem \ref{thm:MainThmA1Fib-Aff}, which
combines several known facts together with new results on equivalence
classes of $\mathbb{A}^{1}$-fibrations $\pi:\mathbb{F}_{n}\setminus B\to\mathbb{A}^{1}$
having an irreducible component of multiplicity two inside their unique
degenerate fiber. 

\section{Preliminaries }

All varieties and schemes considered are defined over a fixed algebraically
closed field $k$ of characteristic zero. 

\subsection{Notations and basic definitions}

We briefly recall basic definitions on SNC divisors and standard properties
of $\mathbb{A}^{1}$-fibrations and $\mathbb{P}^{1}$-fibrations which
we use throughout the paper, see e.g \cite[Chapter 3]{MiyBook} for
the details. 

\vspace{-0.3em}

\subsubsection{\label{subsec:SNC-and-chains}SNC divisors and rationals trees on
smooth surfaces }

\indent\newline\indent (i) Let $X$ be a smooth projective surface.
An \emph{SNC divisor} on $X$ is a curve $B\subset X$ with smooth
irreducible components and ordinary double points only as singularities.
We say that $B$ is \emph{SNC-minimal} if its image by any strictly
birational proper morphism $\tau:X\rightarrow X'$ onto a smooth projective
surface $X'$ with exceptional locus contained in $B$ is not an SNC
divisor. A\emph{ rational tree} on $X$ is an SNC divisor whose irreducible
components are isomorphic to $\mathbb{P}^{1}$ and whose incidence
graph is a tree. A\emph{ rational chain} is a rational tree whose
incidence graph is a chain. We use the notation $B=B_{0}\triangleleft B_{1}\triangleleft\cdots\triangleleft B_{r}$
to indicate a rational chain whose irreducible components $B_{i}$
are ordered in such a way that for $0\leq i<j\leq r$, one has $B_{i}\cdot B_{j}=1$
if $j=i+1$ and $0$ otherwise. The sequence of self-intersections
$[B_{0}^{2},\ldots,B_{r}^{2}]$ is referred to as the \emph{type}
of the oriented rational chain $B$. 

(ii) An \emph{SNC} \emph{completion} of a smooth quasi-projective
surface $V$ is a pair $(X,B)$ consisting of a smooth projective
surface $X$ and an SNC divisor $B\subset V$ such that $X\setminus B\simeq V$.
The completion is said to be SNC-minimal if $B$ is SNC-minimal and
to be smooth if $B$ is smooth. 

\vspace{-0.3em}

\subsubsection{\label{subsec:P1-A1-Fib}Recollection on $\mathbb{A}^{1}$-fibrations
and $\mathbb{P}^{1}$-fibrations on smooth surfaces. }

\indent\newline\indent (i) A \emph{$\mathbb{P}^{1}$-fibration} on
a smooth projective surface $X$ is a morphism $\overline{\rho}:X\rightarrow C$
onto a smooth projective curve $C$ whose generic fiber is isomorphic
to the projective line over the function field of $C$. Every $\mathbb{P}^{1}$-fibration
$\overline{\rho}:X\rightarrow C$ is obtained from a Zariski locally
trivial $\mathbb{P}^{1}$-bundle over $C$ by performing a finite
sequence of blow-ups of points. In particular, every $\mathbb{P}^{1}$-fibration
has a section and its singular fibers are supported by rational trees
on $X$. If $X$ is rational, then for every smooth proper rational
curve $F$ with self-intersection $0$, the complete linear system
$|F|$ of effective divisors on $X$ linearly equivalent to $F$ defines
a $\mathbb{P}^{1}$-fibration $\overline{\rho}_{|F|}:X\rightarrow\mathbb{P}_{k}^{1}$
having $F$ as a smooth fiber. 

(ii) An\emph{ $\mathbb{A}^{1}$-fibration} on a smooth quasi-projective
surface $V$ is a surjective morphism $\rho:V\rightarrow A$ onto
a smooth curve $A$ whose generic fiber is isomorphic to the affine
line over the function field of $A$. Every $\mathbb{A}^{1}$-fibration
is the restriction of a $\mathbb{P}^{1}$-fibration $\overline{\rho}:X\rightarrow C$
over the smooth projective model $C$ of $A$, defined on an SNC completion
$(X,B)$ of $V$with boundary $B=\bigcup_{c\in C\setminus A}F_{c}\cup H\cup\bigcup_{a\in A}G_{a}$
where, $F_{c}=\overline{\rho}^{-1}(c)\simeq\mathbb{P}^{1}$ for every
$c\in C\setminus A$, $H$ is a section of $\overline{\rho}$, and
for every $a\in A$, $G_{a}$ is a union of SNC-minimal rational subtree
of the rational tree $(\overline{\rho}^{-1}(a))_{\mathrm{red}}$.
The pair $(X,B)$ is called a relatively minimal $\mathbb{P}^{1}$-fibered
completion of $\rho:V\to A$. If $\rho:V\to A$ is affine, every nonempty
$G_{a}$ is connected and has an irreducible component intersecting
$H$, and the closure in $X$ of every irreducible component of $\rho^{-1}(a)$
intersects $G_{a}$ transversely in a unique point. In particular,
every irreducible component of $\rho^{-1}(a)$ is isomorphic to $\mathbb{A}^{1}$
when equipped with its reduced structure. A scheme-theoretic closed
fiber of $\rho:V\rightarrow A$ which is not isomorphic to $\mathbb{A}^{1}$
is called \emph{degenerate}. 

(iii) A\emph{ }smooth $\mathbb{A}^{1}$-fibered surface is a pair
$(V,\pi)$ consisting of a smooth quasi-projective surface $V$ and
an $\mathbb{A}^{1}$-fibration $\pi:V\to A$ onto a smooth curve $A$.
The $\mathbb{A}^{1}$-fibration $\pi$ is said to be of affine type
if $A$ is affine and of complete type otherwise. Two $\mathbb{A}^{1}$-fibered
surfaces $(V,\pi:V\to A)$ and $(V',\pi':V'\to A')$ are said to be
\emph{equivalent} if there exist an isomorphism $\Psi:V\to V'$ and
an isomorphism $\psi:A\to A'$ such that $\pi'\circ\Psi=\psi\circ\pi$.

\subsection{\label{subsec:Models}Models of smooth affine surfaces with irreducible
rational boundaries }

We review known basic properties of smooth affine surfaces admitting
smooth completions $(X,B)$ with boundaries $B\cong\mathbb{P}^{1}$.
Recall \cite[Theorem 2]{Goo69} that for such a pair $(X,B)$, the
affineness of $X\setminus B$ implies that $B$ is the support of
an ample effective divisor on $X$. 
\begin{lem}
\label{lem:Pairs}\cite[Proposition 1]{Giz70} A pair $(X,B)$ consisting
of a smooth projective surface $X$ and a divisor $B\cong\mathbb{P}^{1}$such
that $X\setminus B$ is affine is isomorphic to one of the following:

a) $(\mathbb{P}^{2},B)$ where $B$ is either a line $L$ or a smooth
conic $Q$,

b) $(\mathbb{F}_{n},B)$ where $\pi_{n}:\mathbb{F}_{n}=\mathbb{P}(\mathcal{O}_{\mathbb{P}^{1}}\oplus\mathcal{O}_{\mathbb{P}^{1}}(-n))\to\mathbb{P}^{1}$,
$n\geq0$, is a $\mathbb{P}^{1}$-bundle and $B$ is an ample section
of $\pi_{n}$. 
\end{lem}

\begin{proof}
The log-canonical divisor $K_{X}+B$ is not nef since $(K_{X}+B)\cdot B=-2$
by adjunction formula. Given a $K_{X}+B$ extremal smooth rational
curve $R$ on $X$, the conditions $B\cdot R>0$ and $(K_{X}+B)\cdot R<0$
imply that $R^{2}\geq0$. If $R^{2}>0$ then $X$ is a smooth log
del Pezzo surface of Picard rank $1$, hence is isomorphic to $\mathbb{P}^{2}$,
and $B$ is either a line or a smooth conic. If $R^{2}=0$, then the
associated extremal contraction is a Zariski locally trivial $\mathbb{P}^{1}$-bundle
$h:X\to C$ over a smooth projective curve $C$ and $B$ is a section
of $h$. Thus, $C\cong B\cong\mathbb{P}^{1}$ and $(X,h)\cong(\mathbb{F}_{n},\pi_{n})$
for some $n\geq0$. 
\end{proof}
For a pair $(\mathbb{F}_{n},B)$ as in Lemma \ref{lem:Pairs} b),
we denote by $C_{0}\subset\mathbb{F}_{n}$ a section of $\pi_{n}$
with self-intersection $C_{0}^{2}=-n$ and by $F$ a closed fiber
of $\pi_{n}$. Recall \cite[Corollary V.2.18]{HartBook} that for
$m\geq1$, the complete linear system $|C_{0}+mF|$ contains prime
divisors if and only if $m\geq n$. Since $B$ is a section of $\pi_{n}$,
we have $B\sim C_{0}+\frac{1}{2}(B^{2}+n)F$, where $B^{2}\geq n+2$
because $B$ is ample. For a fixed $d\geq2$, the Hirzebruch surfaces
$\mathbb{F}_{n}$ containing an ample section $B$ with $B^{2}=d$
are those of the form $\mathbb{F}_{d-2i}$, $i=1,\ldots,\left\lfloor \tfrac{d}{2}\right\rfloor $,
with $B$ belonging to the complete linear system $|C_{0}+(d-i)F|$. 

Since the divisor class group of $\mathbb{F}_{n}$ is freely generated
by the classes of $F$ and of a section of $\pi_{n}$, the divisor
class group of $\mathbb{F}_{n}\setminus B$ is freely generated by
the class of $F|_{\mathbb{F}_{n}\setminus B}$. A canonical divisor
$K_{\mathbb{F}_{n}}$ of $\mathbb{F}_{n}$ being linearly equivalent
to $-2C_{0}-(n+2)F$, we have $K_{\mathbb{F}_{n}}\sim-2B+(B^{2}-2)F$
and hence $K_{\mathbb{F}_{n}\setminus B}\sim(B^{2}-2)F|_{\mathbb{F}_{n}\setminus B}$.
A result due to Danilov-Gizatullin asserts conversely that the integers
$B^{2}$ are a complete invariant of the isomorphism classes of surfaces
$\mathbb{F}_{n}\setminus B$, namely:
\begin{thm}
\label{thm:DanGiz-Isom}\cite[Theorem 5.8.1]{GiDa77} (see also \cite[Corollary 4.8]{CNR07},
\cite[§3.1]{DubFin14}, \cite{FKZ09} and \cite[Corollary 6.2.4]{FKZ11})
The isomorphism class of the complement of an ample section $B$ in
a Hirzebruch surface $\mathbb{F}_{n}$ depends only on the self-intersection
$B^{2}$ of $B$. 
\end{thm}

\begin{rem}
\label{rem:Torsor} For a pair $(\mathbb{F}_{n},B)$ as in Lemma \ref{lem:Pairs}
b) with $B^{2}=d$, the closed immersion $B\hookrightarrow\mathbb{F}_{n}$
is determined by a surjection $\mathcal{O}_{\mathbb{P}^{1}}\oplus\mathcal{O}_{\mathbb{P}^{1}}(-n)\to\mathcal{L}$
onto an invertible sheaf $\mathcal{L}$ on $\mathbb{P}^{1}$, with
kernel $\mathcal{K}$ isomorphic to $\mathcal{O}_{\mathbb{P}^{1}}(-\frac{1}{2}(d+n))$.
The locally trivial $\mathbb{A}^{1}$-bundle $\nu=\pi_{n}|_{\mathbb{F}_{n}\setminus B}:\mathbb{F}_{n}\setminus B\to\mathbb{P}^{1}$
thus carries the additional structure of a non-trivial torsor under
the line bundle associated to the invertible sheaf $\mathcal{L}^{\vee}\otimes\mathcal{K}\cong\mathcal{O}_{\mathbb{P}^{1}}(-d)$.
Isomorphism classes of such $\mathbb{A}^{1}$-bundles are in one-to-one
correspondence with the elements of the projective space $\mathbb{P}(\mathrm{Ext}^{1}(\mathcal{L},\mathcal{K}))\cong\mathbb{P}(H^{1}(\mathbb{P}^{1},\mathcal{O}_{\mathbb{P}^{1}}(-d)))\cong\mathbb{P}^{d-2}$.
By \cite[Remark 4.8.6]{GiDa77} (see also \cite[§3.1]{DubFin14} and
\cite[Proposition 2]{DubAnnalen15}), every non-trivial $\mathcal{O}_{\mathbb{P}^{1}}(-d)$-torsor
arises as a restriction $\pi_{n}|_{\mathbb{F}_{n}\setminus B}:\mathbb{F}_{n}\setminus B\to\mathbb{P}^{1}$
for some pair $(\mathbb{F}_{n},B)$ as in Lemma \ref{lem:Pairs} b)
with $B^{2}=d$. 
\end{rem}

\begin{example}
\label{exa:Models} The pairs $(X,B)$ of Lemma \ref{lem:Pairs} a)
are unique up to isomorphism. In particular every affine surface isomorphic
to the complement of a smooth conic in $\mathbb{P}^{2}$ is isomorphic
to the complement of the conic $Q_{0}=\{xz+y^{2}=0\}$ in $\mathbb{P}_{[x;y:;z]}^{2}$.
A model of an affine surface $\mathbb{F}_{n}\setminus B$ with $B^{2}=d$
is given for $d=2e\geq2$ by the complement in $\mathbb{F}_{0}=\mathbb{P}_{[u_{0}:u_{1}]}^{1}\times\mathbb{P}_{[v_{0}:v_{1}]}^{1}$
of the section $\Delta_{e}=\{u_{1}^{e}v_{0}-u_{0}^{e}v_{1}=0\}$ of
$\pi_{0}=\mathrm{pr}_{1}$, and for $d=2e+1\geq3$ by the complement
in $\pi_{1}:\mathbb{F}_{1}\to\mathbb{P}^{1}$ viewed as the blow-up
of $\mathbb{P}_{[x:y:z]}^{2}$ at the point $p=[0:1:0]$ of the proper
transform of the rational cuspidal curve $C_{e}=\{yz^{e}+x^{e+1}=0\}$. 
\end{example}

The next examples illustrate some other representatives of isomorphism
classes of affine surfaces $\mathbb{F}_{n}\setminus B$. 
\begin{example}
\cite{DubFin14} For every $d\geq2$ and every pair of integers $p,q\geq1$
such that $p+q=d$, the geometric quotient $S_{d}$ of the smooth
affine threefold $X_{p,q}=\{x^{p}v-y^{q}u=1\}$ in $\mathbb{A}^{4}$
by the free $\mathbb{G}_{m}$-action defined by $\lambda\cdot(x,y,u,v)=(\lambda x,\lambda y,\lambda^{-q}u,\lambda^{-p}v)$
is a representative of the isomorphism class of surfaces of the form
$\mathbb{F}_{n}\setminus B$ such that $B^{2}=d$. Indeed, $S_{d}$
is isomorphic to the complement in the geometric quotient 
\[
\pi_{|p-q|}:\mathbb{F}_{|p-q|}\cong\mathbb{P}(\mathcal{O}_{\mathbb{P}^{1}}(p)\oplus\mathcal{O}_{\mathbb{P}^{1}}(q))\to\mathbb{P}^{1}=\mathrm{Proj}(k[x,y])
\]
of $(\mathbb{A}^{2}\setminus\{0\})\times(\mathbb{A}^{2}\setminus\{0\})$
by the $\mathbb{G}_{m}^{2}$-action $(\lambda,\mu)\cdot(x,y,u,v)=(\lambda x,\lambda y,\lambda^{-p}\mu u,\lambda^{-q}\mu v)$
of the ample section $B$ of $\pi_{|p-q|}$ with self-intersection
$d$ determined by the vanishing of the polynomial $x^{p}v-y^{q}u$
of bi-homogeneous degree $(0,1)$. 
\end{example}

\begin{example}
\cite{DubAnnalen15} For every $d\geq2$, the surface $W_{d}$ in
$\mathbb{A}^{4}=\mathrm{Spec}(k[x_{1},x_{2},x_{2},x_{4}])$ defined
by the equations 
\[
x_{1}x_{3}-x_{2}(x_{2}+1)=0,\;x_{2}^{d-2}x_{4}-x_{3}^{d-1}=0,\;x_{1}^{d-2}x_{4}-(x_{2}+1)^{d-2}x_{3}=0
\]
is a representative of the isomorphism class of surfaces of the form
$\mathbb{F}_{n}\setminus B$ such that $B^{2}=d$. Indeed, the morphism
$\nu:W_{d}\to\mathbb{P}^{1}$, $(x_{1},x_{2},x_{3},x_{4})\mapsto[x_{1}:x_{2}+1]=[x_{2}:x_{3}]$
is a locally trivial $\mathbb{A}^{1}$-bundle with local trivializations
\[
\nu^{-1}(\mathbb{P}^{1}\setminus[0:1])\cong\mathrm{Spec}(k[w][x_{4}])\quad\textrm{and}\quad\nu^{-1}(\mathbb{P}^{1}\setminus[1:0])\cong\mathrm{Spec}(k[w'][x_{1}])
\]
and gluing isomorphism $(w,x_{4})\mapsto(w',x_{1})=(w^{-1},w^{d}x_{4}-w^{d-1})$,
hence is a torsor under the line bundle associated to $\mathcal{O}_{\mathbb{P}^{1}}(-d)$.
By Remark \ref{rem:Torsor}, the surface $W_{d}$ isomorphic to $\mathbb{F}_{n}\setminus B$
for some pair $(\mathbb{F}_{n},B)$ as in Lemma \ref{lem:Pairs} b)
with $B^{2}=d$. 

The surface $W_{2}$ is isomorphic to the surface in $\mathbb{A}^{3}=\mathrm{Spec}(k[x,y,z])$
given by the equation $xy=z(z+1)$. For $d\geq3$, the morphism $W_{d}\to\mathbb{A}^{3}$,
$(x_{1},x_{2},x_{3},x_{4})\mapsto(x_{1},x_{4},x_{2})$ has image equal
to the non-normal surface $V_{d}$ given by the equation $x^{d-1}y=z(z+1)^{d-1}$
and the induced morphism $\nu_{d}:W_{d}\to V_{d}$ is finite and birational,
expressing $W_{d}$ as the normalization of $V_{d}$. This recovers
the other description of representative of surfaces $\mathbb{F}_{n}\setminus B$
such that $B^{2}=d$ as normalizations of surfaces of the form $V_{d}$
given in \cite[§1.0.8]{FKZ11}. 
\end{example}

\section{\label{sec:A1-Fib-P1}Families of $\mathbb{A}^{1}$-fibrations of
complete type}

Since they have torsion class groups, the affine plane $\mathbb{A}^{2}=\mathbb{P}^{2}\setminus L$
and the complements of smooth conics in $\mathbb{P}^{2}$ do not admit
$\mathbb{A}^{1}$-fibrations over complete curves. In contrast, a
surface $\mathbb{F}_{n}\setminus B$ admits a smooth $\mathbb{A}^{1}$-fibration
$\pi_{n}|_{\mathbb{F}_{n}\setminus B}:\mathbb{F}_{n}\setminus B\to\mathbb{P}^{1}$.
In this section, we are interested in the properties of certain $\mathbb{A}^{1}$-fibrations
$\mathbb{F}_{n}\setminus B\to\mathbb{P}^{1}$ with multiple fibers. 
\begin{lem}
\label{lem:Degenerate-Fiber-subscheme}Let $(\mathbb{F}_{n},B)$ be
a pair as in Lemma \ref{lem:Pairs} b), let $q$ be a point of $B$
and let $m\geq1$. Then the linear subsystem $\mathcal{Z}_{q}(m)$
of the complete linear system $|B+mF|$ consisting of divisors intersecting
$B$ with multiplicity $B^{2}+m$ at $q$ has dimension $m$. Furthermore,
the open subset $\mathcal{U}_{q}(m)$ of $\mathcal{Z}_{q}(m)$ consisting
of prime divisors is Zariski dense. 
\end{lem}

\begin{proof}
Put $d=B^{2}$ and $F_{q}=\pi_{n}^{-1}(\pi_{n}(q))$. Let $\mathcal{I}_{B}\cong\mathcal{O}_{\mathbb{F}_{n}}(-B)$
denote the ideal sheaf of $B$ and consider the long exact sequence
of cohomology associated to the short exact sequence of coherent sheaves
\[
0\to\mathcal{I}_{B}\otimes\mathcal{O}_{\mathbb{F}_{n}}(B+mF)\cong\mathcal{O}_{\mathbb{F}_{n}}(mF)\to\mathcal{O}_{\mathbb{F}_{n}}(B+mF)\to\mathcal{O}_{\mathbb{F}_{n}}(B+mF)|_{B}\cong\mathcal{O}_{B}(d+m)\to0
\]
Since $H^{1}(\mathbb{F}_{n},\mathcal{O}_{\mathbb{F}_{n}}(mF))=0$,
the map $H^{0}(\mathbb{F}_{n},\mathcal{O}_{\mathbb{F}_{n}}(B+mF))\to H^{0}(B,\mathcal{O}_{B}(d+m))$
is surjective, with kernel of dimension $\dim H^{0}(\mathbb{F}_{n},\mathcal{O}_{\mathbb{F}_{n}}(mF))=m+1$.
It follows that $\mathcal{Z}_{q}(m)\cong\mathbb{P}^{m}$. A singular
element of $\mathcal{Z}_{q}(m)$ decomposes as the sum of a prime
member $B_{m'}$ of the complete linear system $|B+m'F|$, $0\leq m'<m$,
intersecting $B$ with multiplicity $d+m'$ at $q$ and of $(m-m')F_{q}$.
By the same computation as above, these elements form a closed linear
subspace $\mathcal{Z}_{q}(m')\cong\mathbb{P}^{m'}$ of $\mathcal{Z}_{q}(m)$
and so, $\mathcal{U}_{q}(m)=\mathcal{Z}_{q}(m)\setminus\bigcup_{m'=0}^{m-1}\mathcal{Z}_{q}(m')$
is a dense open subset of $\mathcal{Z}_{q}(m)$. 
\end{proof}
\begin{example}
\label{exa:BvS-Curves}Let $\mathbb{F}_{0}=\mathbb{P}_{[u_{0}:u_{1}]}^{1}\times\mathbb{P}_{[v_{0}:v_{1}]}^{1}$,
$\Delta=\{u_{1}v_{0}-u_{0}v_{1}=0\}$ and let $q=([0:1],[0:1])$.
For every $m\geq1$, denote by $\mathcal{V}_{m}\subset k[t]$ the
$m$-dimensional vector space of monic polynomials of degree $m$.
Writing $P(u,v)=p(\tfrac{v}{u})u^{m}$ for the homogenization of a
polynomial $p(t)\in k[t]$, the map which associates to $p\in\mathcal{V}_{m}$
the section 

\[
B_{m,p}=\{u_{0}P(u_{0},u_{1})v_{1}-(u_{0}^{m+1}+u_{1}P(u_{0},u_{1}))v_{0}=0\}
\]
 of $\pi_{0}=\mathrm{pr}_{1}$ defines an open immersion $\mathcal{V}_{m}\to\mathcal{U}_{q}(m)\subset\mathcal{Z}_{q}(m)\cong\mathbb{P}^{m}$.
These curves $B_{m,p}$ were considered by Blanc-van Santen \cite[Section 2]{BlavS19}
for the fact that their restrictions $B_{m,p}\cap(\mathbb{F}_{0}\setminus\Delta)\cong\mathbb{A}^{1}$
provide examples of non-equivalent closed embeddings of the affine
line into the smooth affine quadric surface $\mathbb{F}_{0}\setminus\Delta$. 
\end{example}

Let $(\mathbb{F}_{n},B)$ be a pair as in Lemma \ref{lem:Pairs} b),
let $m\geq2$ and let $B_{m}$ be a section of $\pi_{n}$ corresponding
to a closed point of the scheme $\mathcal{U}_{q}(m)$ for some $q\in B$.
Let $\mathcal{L}_{q,B_{m}}\subset|B+mF|$ be the pencil generated
by the divisors $B_{m}$ and $B+mF_{q}$. 
\begin{lem}
\label{lem:L_qB-members}Every member of $\mathcal{L}_{q,B_{m}}$
other than $B+mF_{q}$ is a smooth rational curve. 
\end{lem}

\begin{proof}
Every divisor in the complete linear system $|B+mF|$ has self-intersection
$B^{2}+2m$. The minimal resolution $\tau:\tilde{\mathbb{F}}_{n}\to\mathbb{F}_{n}$
of the rational map $\gamma:\mathbb{F}_{n}\dashrightarrow\mathbb{P}^{1}$
defined by $\mathcal{L}_{q,B_{m}}$ is obtained by performing $B^{2}+2m$
successive blow-ups with center at $q$ on the successive proper transforms
of $B_{m}$, with exceptional divisors $E_{1},\ldots,E_{B^{2}+2m}$.
Letting $\tau_{*}^{-1}B_{m}$ be the proper transform of $B_{m}$,
the composition $\tilde{\gamma}=\gamma\circ\tau:\tilde{\mathbb{F}}_{n}\to\mathbb{P}^{1}$
is the $\mathbb{P}^{1}$-fibration defined by the complete linear
system $|\tau_{*}^{-1}B_{m}|$. The total transform of $B_{m}$ is
a rational chain $E_{1}\triangleleft\cdots\triangleleft E_{B^{2}+2m}\triangleleft\tau_{*}^{-1}B_{m}$,
where $E_{B^{2}+2m}$ is a section of $\tilde{\gamma}$. Since every
singular member of $\mathcal{L}_{q,B_{m}}$ is the sum of a prime
member $B'$ of the linear system $|B+m'F|$ for some $0\leq m'<m$
and of $(m-m')F_{q}$, every fiber of $\tilde{\gamma}$ other than
that containing the proper transform of $B\cup F_{q}$ is smooth.
This implies that every member of $\mathcal{L}_{q,B_{m}}$ other than
$B+mF_{q}$ is a smooth rational curve. 
\end{proof}
For every $q\in B$, the space of pencils $\mathcal{L}_{q,B_{m}}$
identifies with a dense open subset $\mathcal{S}_{q}(m)$ of the projective
space $\mathbb{P}^{m-1}$ of lines passing through the point of $\mathcal{Z}_{q}(m)\setminus\mathcal{U}_{q}(m)$
corresponding to the reducible divisor $B+mF_{q}$ and a point of
$\mathcal{U}_{q}(m)$. The linear action on $H^{0}(\mathbb{F}_{n},\mathcal{O}_{\mathbb{F}_{n}}(B+mF))$
of the algebraic subgroup $\mathrm{Aut}(\mathbb{F}_{n},B\cup F_{q})$
of automorphisms of $\mathbb{F}_{n}$ preserving $B\cup F_{q}$ induces
an action of $\mathrm{Aut}(\mathbb{F}_{n},B\cup F_{q})$ on $\mathcal{S}_{q}(m)$.
On the other hand, the rational map $\gamma_{q,B_{m}}:\mathbb{F}_{n}\dashrightarrow\mathbb{P}^{1}$
associated to a pencil $\mathcal{L}_{q,B_{m}}$ restricts to a surjective
$\mathbb{A}^{1}$-fibration $\delta_{q,B_{m}}:\mathbb{F}_{n}\setminus B\to\mathbb{P}^{1}$
having $F_{q}\cap(\mathbb{F}_{n}\setminus B)\cong\mathbb{A}^{1}$
as a unique degenerate fiber, of multiplicity $m$. We have the following
characterization:
\begin{lem}
\label{lem:Pencil-Aut-Orbit-Equivalence} Let $(\mathbb{F}_{n},B)$
be a pair as in Lemma \ref{lem:Pairs} b), let $m\geq3$ and let $B_{m}$
and $B_{m}'$ be sections of $\pi_{n}:\mathbb{F}_{n}\to\mathbb{P}^{1}$
corresponding to points of the scheme $\mathcal{U}_{q}(m)$ for some
$q\in B$. Then the $\mathbb{A}^{1}$-fibered surfaces $(\mathbb{F}_{n}\setminus B,\delta_{q,B_{m}})$
and $(\mathbb{F}_{n}\setminus B,\delta_{q,B_{m}'})$ are equivalent
if and only if the pencils $\mathcal{L}_{q,B_{m}}$ and $\mathcal{L}_{q,B_{m}'}$
belong to the same $\mathrm{Aut}(\mathbb{F}_{n},B\cup F_{q})$-orbit. 
\end{lem}

\begin{proof}
Let $\Psi:(\mathbb{F}_{n}\setminus B,\delta_{q,B_{m}})\to(\mathbb{F}_{n}\setminus B,\delta_{q,B_{m}'})$
be an equivalence of $\mathbb{A}^{1}$-fibered surfaces, let $\bar{\Psi}$
be its extension to a birational automorphism of $\mathbb{F}_{n}$
and let $\mathbb{F}_{n}\stackrel{\sigma}{\leftarrow}Y\stackrel{\sigma'}{\to}\mathbb{F}_{n}$
be the minimal resolution of $\bar{\Psi}$. If $\sigma'$ is not an
isomorphism, then $\sigma$ is not an isomorphism and $\Psi$ contracts
$B$ onto a point. By the minimality assumption, the proper transform
$\sigma_{*}^{-1}(B)$ of $B$ is the only $\sigma'$-exceptional $(-1)$-curve
contained in $\sigma^{-1}(B)$. It follows that $\sigma$ has $B^{2}+1$
proper or infinitely near base points on $B$ and hence, since $B\cdot B_{m}=B^{2}+m$,
that $\sigma_{*}^{-1}B_{m}\cdot\sigma_{*}^{-1}(B)\geq m-1\geq2$.
But then, $\sigma'(\sigma_{*}^{-1}B_{m})$ is an irreducible singular
member of $\mathcal{L}_{q,B_{m}'}$, which is impossible by Lemma
\ref{lem:L_qB-members}. Thus, $\sigma'$ is an isomorphism and $\bar{\Psi}$
is an automorphism of $\mathbb{F}_{n}$, which preserves $B$ and
the closure $F_{q}$ of the unique common degenerate fiber $F_{q}\cap(\mathbb{F}_{n}\setminus B)$
of $\delta$ and $\delta'$. Furthermore, $\bar{\Psi}$ maps $B_{m}$
onto a certain smooth member of $\mathcal{L}_{q,B_{m}'}$, hence maps
$\mathcal{L}_{q,B_{m}}$ onto $\mathcal{L}_{q,B_{m}'}$. The converse
implication is clear. 
\end{proof}
\begin{rem}
For every $m\geq2$ and every point $q\in B$, one can find distinct
points $B_{m}$ and $B_{m}'$ in the scheme $\mathcal{U}_{q}(m)$
such that the pencils $\mathcal{L}_{q,B_{m}}$ and $\mathcal{L}_{q,B_{m}'}$
have distinct general members. The associated $\mathbb{A}^{1}$-fibrations
$\delta_{q,B_{m}}$ and $\delta_{q,B_{m}'}$ have distinct general
fibers but share $F_{q}\cap(\mathbb{F}_{n}\setminus B)$ as a degenerate
fiber. This contrasts with $\mathbb{A}^{1}$-fibrations of affine
type for which no curve can be contained simultaneously in fibers
of two $\mathbb{A}^{1}$-fibrations with distinct general fibers,
see \cite[Corollary 2.22]{Dub05}.
\end{rem}

\begin{prop}
\label{prop:UncountableA1FIbComplete}Let $(\mathbb{F}_{n},B)$ be
a pair as in Lemma \ref{lem:Pairs} b). Then for every $m\geq4$,
there exist infinitely many equivalence classes of $\mathbb{A}^{1}$-fibrations
$\pi:\mathbb{F}_{n}\setminus B\to\mathbb{P}^{1}$ with a unique degenerate
fiber of multiplicity $m$. 
\end{prop}

\begin{proof}
By Theorem \ref{thm:DanGiz-Isom}, it suffices to construct such families
from the two pairs $(\mathbb{F}_{n},B)=(\mathbb{F}_{0},\Delta_{e})$
and $(\mathbb{F}_{1},\sigma_{*}^{-1}C_{e})$ of Example \ref{exa:Models}.
If $B^{2}=2e\geq2$, let $q=([1:0],[0:1])\in\Delta_{e}=\{u_{1}^{e}v_{0}-u_{0}^{e}v_{1}=0\}\subset\mathbb{F}_{0}$.
The group $\mathrm{Aut}(\mathbb{F}_{0},\Delta_{e}\cup F_{q})$ is
isomorphic to the affine group $\mathbb{G}_{m}\ltimes\mathbb{G}_{a}$
acting by 
\[
(\lambda,t)\cdot([u_{0}:u_{1}],[v_{0}:v_{1}])=([\lambda u_{0}+tu_{1}:u_{1}],[v_{0}+\lambda^{-1}tv_{1}:\lambda^{-1}v_{1}])
\]
if $e=1$ and for every $e\geq2$ to the group $\mathbb{G}_{m}$ acting
by $\lambda\cdot([u_{0}:u_{1}],[v_{0}:v_{1}])=([\lambda u_{0}:u_{1}],[v_{0}:\lambda^{-e}v_{1}])$.
If $B^{2}=2e+1\geq3$, viewing $\mathbb{F}_{1}$ as the blow-up $\sigma:\mathbb{F}_{1}\to\mathbb{P}_{[x:y:z]}^{2}$
of the point $p=[0:1:0]$ with exceptional divisor $C_{0}$, let $q$
be the intersection point of $\sigma_{*}^{-1}C_{e}$ with the proper
transform of the tangent line $L=\{z=0\}$ to $C_{e}=\{yz^{e}+x^{e+1}=0\}$
at $p$. The group $\mathrm{Aut}(\mathbb{F}_{0},\sigma_{*}^{-1}C_{e}\cup F_{q})$
is then isomorphic to the group $\mathrm{Aut}(\mathbb{P}^{2},C_{e}\cup L)$.
The latter is isomorphic to $\mathbb{G}_{m}\ltimes\mathbb{G}_{a}$
acting by $(\lambda,t)\cdot[x:y:z]=[\lambda x+tz:\lambda^{2}y-2\lambda tx-t^{2}z:z]$
if $e=1$ and for every $e\geq2$ to $\mathbb{G}_{m}$ acting by $\lambda\cdot[x:y:z]=[\lambda x:\lambda^{e+1}y:z]$. 

In both cases, the $\mathrm{Aut}(\mathbb{F}_{n},B\cup F_{q})$-orbit
of a point of the open subset $\mathcal{S}_{q}(m)\subset\mathbb{P}^{m-1}$
is at most $2$-dimensional. Since $m-1\geq3$, the set-theoretic
orbit space $\mathcal{S}_{q}(m)/\mathrm{Aut}(\mathbb{F}_{n},B\cup F_{q})$
is infinite and the assertion follows from Lemma \ref{lem:Pencil-Aut-Orbit-Equivalence}. 
\end{proof}

\section{Equivalence classes of $\mathbb{A}^{1}$-fibrations of affine type }

\subsection{\label{subsec:Special-pencils}Special pencils of rational curves
and associated $\mathbb{A}^{1}$-fibrations of affine type }

Let $(X,B)$ be a pair as in Lemma \ref{lem:Pairs}. For every point
$q\in B$, denote by $\mathcal{P}_{q}$ the linear subsystem of the
complete linear system $|B|$ on $X$ consisting of curves with local
intersection number with $B$ at $q$ equal to $B^{2}$. If $(X,B)\cong(\mathbb{P}^{2},L)$
where $L$ is a line, then $\mathcal{P}_{q}$ is simply the pencil
of lines through $q$. More generally, if $B^{2}\geq2$ then the same
type of computation as in the proof of Lemma \ref{lem:L_qB-members}
implies that $\mathcal{P}_{q}$ is a pencil. The minimal resolution
$\sigma:\tilde{X}\to X$ of the rational map $\rho_{q}:X\dashrightarrow\mathbb{P}^{1}$
defined by $\mathcal{P}_{q}$ is obtained by performing $B^{2}$ successive
blow-ups with center at $q$ on the successive proper transforms of
$B$, with respective exceptional divisor $E_{1},\ldots,E_{B^{2}}$.
The total transform of $B$ in $\tilde{X}$ is a rational chain $\sigma_{*}^{-1}B\triangleleft E_{B^{2}}\triangleleft E_{B^{2}-1}\triangleleft\cdots E_{1}$
of type $[0,-1,-2,\ldots,-2]$ and the morphism $\tilde{\rho}{}_{q}=\rho_{q}\circ\sigma:\tilde{X}\to\mathbb{P}^{1}$
is the $\mathbb{P}^{1}$-fibration defined by the complete linear
system $|\sigma_{*}^{-1}B|$. 
\begin{example}
\label{exa:ResoPencilConics} For every point $q$ on a smooth conic
$Q\subset\mathbb{P}^{2}$, $\mathcal{P}_{q}$ is the pencil of conics
intersecting $Q$ with multiplicity $4$ at $q$, generated by $Q$
and twice its tangent line $T_{q}Q$ at $q$. The total transform
of $Q\cup T_{q}Q$ by $\sigma$ is the rational tree \[\xymatrix@R=0.5em @C=0.3em{(\sigma_*^{-1}Q,0) \ar@{}[r]|\triangleleft & (E_4,-1) \ar@{}[r]|\triangleleft & (E_3,-2) \ar@{}[r]|\triangleleft & (E_2,-2) \ar@{-}[d] \ar@{-}[r]& (\sigma_*^{-1}T_qQ,-1) \\ & & & (E_1,-2),}\]   where
the displayed numbers in the parenthesis are the self-intersection
numbers of the corresponding irreducible components. The $\mathbb{P}^{1}$-fibration
$\tilde{\rho}_{q}:\tilde{\mathbb{P}}^{2}\to\mathbb{P}^{1}$ has $\tilde{\rho}_{q}^{-1}(\rho_{q}(T_{q}Q))=\bigcup_{i=1}^{3}E_{i}\cup\sigma_{*}^{-1}T_{q}Q$
as a unique singular fiber.
\end{example}

For pairs $(\mathbb{F}_{n},B)$ as in Lemma \ref{lem:Pairs} b), we
have the following description (see also \cite[Section 3]{DubLam15}): 
\begin{lem}
\label{lem:Special-Pencil-Reso} Let $(\mathbb{F}_{n},B)$ be a pair
as in Lemma \ref{lem:Pairs} b). Then for every point $q\in B$, the
following hold:

a) The pencil $\mathcal{P}_{q}$ has a unique singular member consisting
of a divisor of the form $C+m_{q}F_{q}$, where $C$ is a section
of $\pi_{n}$, $F_{q}=\pi_{n}^{-1}(\pi_{n}(q))$ and $m_{q}\in\{1,\ldots,B^{2}-1\}$. 

b) The $\mathbb{P}^{1}$-fibration $\tilde{\rho}_{q}:\tilde{\mathbb{F}}_{n}\to\mathbb{P}^{1}$
has $\tilde{\rho}_{q}^{-1}(\rho_{q}(C\cup F_{q}))=\bigcup_{i=1}^{B^{2}-1}E_{i}\cup\sigma_{*}^{-1}(C\cup F_{q})$
as a unique singular fiber. 

c) For a general point $q\in B$, the unique singular member of $\mathcal{P}_{q}$
is reduced. 
\end{lem}

\begin{proof}
With the notation of subsection \ref{subsec:Models}, put $d=B^{2}\geq2$
and $B\sim C_{0}+\ell F$ with $\ell=\frac{1}{2}(d+n)\leq d-1$. Since
$d\geq2$ and $E_{d-1}^{2}=-2$, $\tilde{\rho}_{q}^{-1}(\tilde{\rho}_{q}(E_{d-1}))$
is a singular fiber of $\tilde{\rho}_{q}$ and its image by $\sigma$
is a singular member of $\mathcal{P}_{q}$. Since a singular member
of $\mathcal{P}_{q}$ decomposes as the union of a section $C\sim C_{0}+\ell'F$
of $\pi_{n}$ for some $0\leq\ell'<\ell\leq d-1$ intersecting $B$
with multiplicity $d-(\ell-\ell')$ at $q$ and of $(\ell-\ell')F_{q}$,
it follows that $\sigma(\tilde{\rho}_{q}^{-1}(\tilde{\rho}_{q}(E_{d-1})))=C+(\ell-\ell')F_{q}$
is the unique singular member of $\mathcal{P}_{q}$. This proves a)
and b). For assertion c), see \cite[Proposition 4.8.11]{GiDa77} or
\cite[Lemma 3.2]{DubFin14}. 
\end{proof}
For a pencil $\mathcal{P}_{q}$ on a pair $(X,B)$ as in Lemma \ref{lem:Pairs},
the rational map $\rho_{q}:\mathbb{F}_{n}\dashrightarrow\mathbb{P}^{1}$
defined by $\mathcal{P}_{q}$ restricts to an $\mathbb{A}^{1}$-fibration
$\pi_{q}:X\setminus B\to\mathbb{A}^{1}=\mathbb{P}^{1}\setminus\rho_{q}(B)$.
In the case where $(X,B)=(\mathbb{P}^{2},L)$ for some line $L$,
$\pi_{q}$ is a trivial $\mathbb{A}^{1}$-bundle on $\mathbb{A}^{2}=\mathbb{P}^{2}\setminus L$,
and in the case where $(X,B)=(\mathbb{P}^{2},Q)$, $\pi_{q}$ has
a unique degenerate fiber consisting of $T_{q}Q\cap(\mathbb{P}^{2}\setminus Q)\cong\mathbb{A}^{1}$
occurring with multiplicity $2$ (see Example \ref{exa:ResoPencilConics}).
For pairs $(\mathbb{F}_{n},B)$, it follows from Lemma \ref{lem:Special-Pencil-Reso}
a) that $\pi_{q}:\mathbb{F}_{n}\setminus B\to\mathbb{A}^{1}$ has
unique degenerate fiber which is reducible, consisting of the disjoint
union of $C\cap(\mathbb{F}_{n}\setminus B)\cong\mathbb{A}^{1}$ occurring
with multiplicity $1$ and of $F_{q}\cap(\mathbb{F}_{n}\setminus B)\cong\mathbb{A}^{1}$
occurring with a certain multiplicity $m_{q}\in\{1,\ldots,B^{2}-1\}$.
The following lemma shows conversely that for a pair $(X,B)$ as above,
every $\mathbb{A}^{1}$-fibration $\pi:X\setminus B\to\mathbb{A}^{1}$
is induced by a pencil $\mathcal{P}_{q}$ on a suitable smooth completion
$(X',B')$ of $X\setminus B$. 
\begin{lem}
\label{lem:A1-fib-inducedbypencil}Let $(X,B)$ be a pair as in Lemma
\ref{lem:Pairs} and let $\pi:X\setminus B\to\mathbb{A}^{1}$ be an
$\mathbb{A}^{1}$-fibration. Then there exists a smooth completion
$(X',B')$ of $X\setminus B$ by some pair as in Lemma \ref{lem:Pairs},
an isomorphism $\Psi:X\setminus B\to X'\setminus B'$ and a point
$q'\in B'$ such that $\pi=\pi_{q'}\circ\Psi$. 
\end{lem}

\begin{proof}
Note that $X\setminus B$ admits an SNC completion by a rational chain
of type $[0,-1]$ if $(X,B)\cong(\mathbb{P}^{2},L)$ and of type $[0,-1,-2,\ldots,-2]$
with $B^{2}-1\geq1$ curves of self-intersection number $-2$ otherwise.
Now let $(Y,D)$ be a relatively minimal SNC completion of $\pi:X\setminus B\to\mathbb{A}^{1}$
into a $\mathbb{P}^{1}$-fibration $\bar{\pi}:Y\to\mathbb{P}^{1}$
as in subsection \ref{subsec:P1-A1-Fib} (ii). By \cite[Proposition 2.15]{Dub05},
$D$ is a rational chain $F_{\infty}\triangleleft H\triangleleft E$,
where $F_{\infty}\cong\mathbb{P}^{1}$ is the fiber of $\bar{\pi}$
over the point $\mathbb{P}^{1}\setminus\mathbb{A}^{1}$, $H$ is a
section of $\bar{\pi}$ and $E$ is either the empty divisor or a
rational chain $E_{1}\triangleleft\cdots\triangleleft E_{d-1}$ consisting
of curves with self-intersection number $\leq-2$ contained in a fiber
of $\bar{\pi}$. By making elementary transformations consisting of
the blow-up of a point of $F_{\infty}$ followed by the contraction
of the proper transform of $F_{\infty}$, we can further assume from
the beginning that $H^{2}=-1$. By \cite[Corollary 2]{GiDa75} (see
also \cite[Corollary 3.32]{FKZ07} or \cite[Corollary 3.2.3]{BlanDub11}),
the number of irreducible components of $E$ and their self-intersection
numbers are independent on $(Y,D)$. Thus, $D$ is a chain of one
of the types listed above and so, letting $\tau:Y\to X'$ be the contraction
of the subchain $H\triangleleft E$ onto a smooth point $q'\in B'=\tau(F_{\infty})\cong\mathbb{P}^{1}$,
we obtain a smooth completion $(X',B')$ of $X\setminus B$ and an
isomorphism $\Psi=\tau|_{Y\setminus D}:X\setminus B\cong Y\setminus D\to X'\setminus B'$
such that $\pi=\pi_{q'}\circ\Psi$. 
\end{proof}
\begin{cor}
\label{cor:UniquenessA1Fib}The affine plane $\mathbb{A}^{2}$, the
complement $\mathbb{P}^{2}\setminus Q$ of a smooth conic $Q\subset\mathbb{P}^{2}$
and the affine quadric surface $\mathbb{P}^{1}\times\mathbb{P}^{1}\setminus\Delta$
all have a unique equivalence class of $\mathbb{A}^{1}$-fibrations
over $\mathbb{A}^{1}$. 
\end{cor}

\begin{proof}
In each case, given an $\mathbb{A}^{1}$-fibration $\pi:X\setminus B\to\mathbb{A}^{1}$,
Lemma \ref{lem:A1-fib-inducedbypencil} provides a smooth completion
$(X',B')$ of $X\setminus B$ such that $\pi=\pi_{q'}\circ\Psi$ for
some isomorphism $\Psi:X\setminus B\to X'\setminus B'$ and some point
$q'\in B'$. If $(X,B)=(\mathbb{P}^{2},L)$ or $(\mathbb{P}^{2},Q)$
then $X'=\mathbb{P}^{2}$ and $B'$ is respectively a line $L'$ or
a smooth conic $Q'$. If $(X,B)=(\mathbb{P}^{1}\times\mathbb{P}^{1},\Delta)$
then $X'=\mathbb{P}^{1}\times\mathbb{P}^{1}$ and $B'$ is a prime
divisor of type $(1,1)$. The assertion then follows from the fact
that in each case, the automophism group of the surface $X'=X$ acts
transitively on the set of pairs $(B',q')$. 
\end{proof}
\begin{cor}
\label{cor:GizDanSurface-A1FiberType} Every $\mathbb{A}^{1}$-fibration
$\pi:\mathbb{F}_{n}\setminus B\to\mathbb{A}^{1}$ on an affine surface
$\mathbb{F}_{n}\setminus B$ has a unique degenerate fiber which consists
of the disjoint union of a reduced irreducible component and an irreducible
component of multiplicity $m\in\{1,\ldots,B^{2}-1\}$. 
\end{cor}

In contrast to Corollary \ref{cor:UniquenessA1Fib}, the following
lemma, whose proof reproduces that of \cite[Theorem 16.2.1]{DubThesis}
(in french), implies that every affine surface $\mathbb{F}_{n}\setminus B$
with $B^{2}\geq3$ admits more than one equivalence class of $\mathbb{A}^{1}$-fibrations
over $\mathbb{A}^{1}$. 
\begin{lem}
\label{lem:A1Fib-Mult-Existence}Let $(\mathbb{F}_{n},B)$ be a pair
as in Lemma \ref{lem:Pairs} b). Then for every integer $m\in\{1,\ldots,B^{2}-1\}$,
there exists an $\mathbb{A}^{1}$-fibration $\pi_{m}:\mathbb{F}_{n}\setminus B\to\mathbb{A}^{1}$
whose degenerate fiber has a reduced component and a component of
multiplicity $m$.\footnote{Iin particular, there exist at least $B^{2}-1$ equivalence classes
of $\mathbb{A}^{1}$-fibrations over $\mathbb{A}^{1}$ on $\mathbb{F}_{n}\setminus B$.
The lower bound $\left\lfloor \tfrac{B^{2}-1}{2}\right\rfloor $ was
discovered earlier by Peter Russell (unpublished) and was proven by
Flenner-Kaliman-Zaidenberg \cite[Corollary 5.16 a)]{FKZ07-2} using
a closely related construction.} 
\end{lem}

\begin{proof}
We use the notation introduced in subsection \ref{subsec:Models}
and put $d=B^{2}\geq2$. Let $n'=d-2i$ for some $i=1,\ldots\left\lfloor \tfrac{d}{2}\right\rfloor $,
let $C_{n'}$ be a prime member of the complete linear system $|C_{0}+n'F|$
and let $q_{0}\in C_{0}$ and $q_{n'}\in C_{n'}$ be a pair of closed
points contained in two different fibers $F_{q_{0}}$ and $F_{q_{n'}}$
of $\pi_{n'}:\mathbb{F}_{n'}\to\mathbb{P}^{1}$. Note that $B'\cdot C_{0}=i$
and $B'\cdot C_{n'}=d-i$ for every member $B'$ of $|C_{0}+(d-i)F|$.
Applying a sequence of $i$ elementary transformations with center
at $q_{0}$ followed by a sequence of $d-i$ elementary transformations
with center at $q_{n'}$ yields a birational map $\beta:\mathbb{F}_{n'}\dashrightarrow\mathbb{F}_{0}=\mathbb{P}^{1}\times\mathbb{P}^{1}$
such that $\pi_{n'}=\mathrm{pr}_{1}\circ\beta$. The composition $\mathrm{pr}_{2}\circ\beta:\mathbb{F}_{n'}\dashrightarrow\mathbb{P}^{1}$
is given by a pencil $\mathcal{L}\subset|C_{0}+(d-i)F|$ whose general
members are sections $B'$ of $\pi_{n'}$ which satisfy $B'\cap C_{0}=q_{0}$
and $B'\cap C_{n'}=q_{n'}$. Since $(B')^{2}=2(d-i)-(d-2i)=d,$ $\mathbb{F}_{n'}\setminus B'$
is isomorphic to $\mathbb{F}_{n}\setminus B$ by Theorem \ref{thm:DanGiz-Isom}.
On other hand, the pencil $\mathcal{P}_{q_{0}}$ has a unique singular
member equal to $C_{0}+(d-i)F_{q_{0}}$ whereas the pencil $\mathcal{P}_{q_{n'}}$
has a unique singular member equal to $C_{n'}+iF_{q_{n'}}$. The degenerate
fibers of the associated $\mathbb{A}^{1}$-fibrations $\pi_{q_{0}}:\mathbb{F}_{n'}\setminus B'\to\mathbb{A}^{1}$
and $\pi_{q_{n'}}:\mathbb{F}_{n'}\setminus B'\to\mathbb{A}^{1}$ have
$F_{q_{0}}\cap(\mathbb{F}_{n'}\setminus B')$ and $F_{q_{n'}}\cap(\mathbb{F}_{n'}\setminus B')$
as irreducible components of multiplicity $d-i$ and $i$ respectively.
Since $i$ ranges from $1$ to $\left\lfloor \tfrac{d}{2}\right\rfloor $,
the assertion follows. 
\end{proof}

\subsection{Some classes of $\mathbb{A}^{1}$-fibrations of affine type on surfaces
$\mathbb{F}_{n}\setminus B$ }

By Corollary \ref{cor:GizDanSurface-A1FiberType} and Lemma \ref{lem:A1Fib-Mult-Existence},
the classification of equivalences classes of $\mathbb{A}^{1}$-fibrations
$\pi:\mathbb{F}_{n}\setminus B\to\mathbb{A}^{1}$ is divided into
that of each type according to the multiplicity $m\in\{1,\ldots,B^{2}-1\}$
of the possibly non-reduced irreducible component of their unique
degenerate fiber. Hereafter, we first recall known results on the
two extremal cases: $\mathbb{A}^{1}$-fibrations with a component
of maximal multiplicity $B^{2}-1$ on the one hand, and smooth $\mathbb{A}^{1}$-fibrations
on the other hand. We then proceed to the study of equivalence classes
of $\mathbb{A}^{1}$-fibrations with a component of multiplicity two
in their unique degenerate fiber. 

\subsubsection{Equivalence classes of $\mathbb{A}^{1}$-fibrations with maximal
multiplicity}
\begin{prop}
\label{lem:OrbitMaximalPair} \label{cor:UniqueMaximal}For every
pair $(\mathbb{F}_{n},B)$ as in Lemma \ref{lem:Pairs} b), the affine
surface $\mathbb{F}_{n}\setminus B$ has a unique equivalence class
of $\mathbb{A}^{1}$-fibration $\pi:\mathbb{F}_{n}\setminus B\to\mathbb{A}^{1}$
with a degenerate fiber containing an irreducible component of multiplicity
$B^{2}-1$. 
\end{prop}

\begin{proof}
A pair $(\mathbb{F}_{n},B)$ with $B^{2}=d+2\geq2$ such that $B$
contains a point $q$ for which the singular member of the pencil
$\mathcal{P}_{q}$ has the form $C+(d+1)F_{q}$ for some irreducible
curve $C$ is necessarily equal to $(\mathbb{F}_{d},B)$ for some
section $B\sim C_{0}+(d+1)F$ of $\pi_{d}$ intersecting $C_{0}$
transversely at the point $q$, the curve $C$ being then equal to
$C_{0}$. Then the assertion follows from the fact that the group
$\mathrm{Aut}(\mathbb{F}_{d})$ acts transitively on the set of sections
$B\sim C_{0}+(d+1)F$ of $\pi_{d}$. Let us recall the argument. Since
the restriction homomorphism $\mathrm{Aut}(\mathbb{F}_{d},C_{0})\to\mathrm{Aut}(C_{0})$
is surjective, it suffices to show that for some chosen fiber $F_{0}$
of $\pi_{d}$, the action of $\mathrm{Aut}(\mathbb{F}_{d},C_{0}\cup F_{0})$
on the set of sections $B_{0}\sim C_{0}+(d+1)F$ intersecting $C_{0}$
at the point $F_{0}\cap C_{0}$ is transitive. Identifying $\mathbb{F}_{d}\setminus(C_{0}\cup F_{0})$
to $\mathbb{A}^{2}=\mathrm{Spec}(k[x,y])$ in such a way that $\pi_{d}|_{\mathbb{A}^{2}}=\mathrm{pr}_{x}$
and that the closures in $\mathbb{F}_{d}$ of the level sets of $y$
are sections of $\pi_{d}$ linearly equivalent to $C_{0}+dF$, these
sections $B_{0}$ are the closures in $\mathbb{F}_{d}$ of curves
$\Gamma_{p}\subset\mathbb{A}^{2}$ defined by equations of the form
$y=p(x)$ where $p(x)\in k[x]$ is a polynomial of degree $d+1$.
Since every automorphism of $\mathbb{A}^{2}$ of the form $(x,y)\mapsto(\lambda x+\mu,\nu y+r(x))$,
where $r(x)\in k[x]$ is a polynomial of degree at most $d$, extends
to an element of $\mathrm{Aut}(\mathbb{F}_{d},C_{0}\cup F_{0})$,
it follows that every section $B_{0}$ belongs to the $\mathrm{Aut}(\mathbb{F}_{d},C_{0}\cup F_{0})$-orbit
of the closure of the curve $\Gamma_{x^{d+1}}=\{y=x^{d+1}\}$. 
\end{proof}

\subsubsection{Equivalence classes of smooth $\mathbb{A}^{1}$-fibrations }

The following lemma is a reformulation of \cite[Lemma 5.5.5]{GiDa77},
which appeared, stated in a different language, in the Appendix of
\cite{DubFin14}. 
\begin{lem}
\label{lem:Uniqueness-Reduced-A1Fib}Let $(\mathbb{F}_{n},B)$ be
a pair as in Lemma \ref{lem:Pairs} b) and let $q\in B$ be a point
such that the singular member of the pencil $\mathcal{P}_{q}$ is
reduced. Then the isomorphism type of the $\mathbb{A}^{1}$-fibered
surface $\pi_{q}:\mathbb{F}_{n}\setminus B\to\mathbb{A}^{1}$ depends
only on the integer $B^{2}$. 
\end{lem}

\begin{proof}
Put $d=B^{2}$ and $S=\mathbb{F}_{n}\setminus B$. The singular member
of $\mathcal{P}_{q}$ has the form $C+F_{q}$ where $C$ is a prime
member of the complete linear system $|B-F_{q}|$. Without loss of
generality, we can fix an isomorphism $A=\mathbb{P}^{1}\setminus\rho_{q}(B)\cong\mathrm{Spec}(k[x])$
so that $\rho_{q}(C\cup F_{q})=\{0\}\in A$. With the notation of
subsection \ref{subsec:Special-pencils}, the total transform of $B\cup C\cup F_{q}$
in the minimal resolution $\sigma:\tilde{\mathbb{F}}_{n}\to\mathbb{F}_{n}$
of the rational map map $\rho_{q}:\mathbb{F}_{n}\dashrightarrow\mathbb{P}^{1}$
defined by $\mathcal{P}_{q}$ is a rational tree of the form \[\xymatrix@R=1.3em @C=0.9em{(\sigma_*^{-1}B,0) \ar@{}[r]|\triangleleft & (E_d,-1)\ar@{}[r]|\triangleleft & (E_{d-1},-2) \ar@{-}[d] \ar@{}[r]|\triangleleft & (E_{d-2},-2) \ar@{}[r]|\triangleleft & \cdots\cdots \ar@{}[r]|\triangleleft & (E_{1},-2) \ar@{-}[r] & (\sigma_*^{-1}F_q,-1) \\ & & (\sigma_*^{-1}C,-1).}\] 
Furthermore, the singular fiber $\tilde{\rho}_{q}^{-1}(\rho_{q}(C\cup F_{q}))=\bigcup_{i=1}^{d-1}E_{1}\cup\sigma_{*}^{-1}C\cup\sigma_{*}^{-1}F_{q}$
of the $\mathbb{P}^{1}$-fibration $\tilde{\rho}_{q}=\rho_{q}\circ\sigma:\tilde{\mathbb{F}}_{n}\to\mathbb{P}^{1}$
is reduced. By contracting successively $\sigma_{*}^{-1}F_{q}$, $E_{1},\ldots,E_{d-2}$
and $\sigma_{*}^{-1}C$, we get a birational morphism $\overline{\tau}:\tilde{\mathbb{F}}_{n}\rightarrow\mathbb{F}_{1}$
of $\mathbb{P}^{1}$-fibered surfaces over $\mathbb{P}^{1}$. The
later restricts to a morphism 
\[
\tau:S\cong\tilde{\mathbb{F}}_{n}\setminus\sigma^{-1}(B)\longrightarrow\mathbb{F}_{1}\setminus\bar{\tau}(\sigma_{*}^{-1}(B)\cup E_{d})\simeq A\times\mathbb{A}^{1}
\]
of schemes over $A$, inducing an isomorphism $S\setminus(C\cup F_{q})\stackrel{\sim}{\rightarrow}A\setminus\{0\}\times\mathbb{A}^{1}$
and contracting $C\cap S$ and $F_{q}\cap S$ to a pair of distinct
points supported on $\{0\}\times\mathbb{A}^{1}$. One can choose a
coordinate on the second factor of $A\times\mathbb{A}^{1}$ and a
pair of isomorphisms of $A$-schemes $S\setminus F_{q}\simeq A\times\mathrm{Spec}(k[u])$
and $S\setminus C\simeq A\times{\rm Spec}([u'])$ so that the restrictions
of $\tau$ to $S\setminus F_{q}$ and $S\setminus C$ coincide respectively
with the morphisms 
\[
S\setminus F_{q}\rightarrow A\times\mathbb{A}^{1},\;(x,u)\mapsto(x,xu+1)\quad\textrm{and}\quad S\setminus C\rightarrow A\times\mathbb{A}^{1},\;\left(x,u'\right)\mapsto(x,x^{d-1}u').
\]
Thus, $\pi_{q}:S\to A$ is $A$-isomorphic to the surface $W_{d}$
obtained by gluing two copies $U_{\pm}=\mathrm{Spec}(k[x][v_{\pm}])$
of $A\times\mathbb{A}^{1}$ along the open subset $(A\setminus\{0\})\times\mathbb{A}^{1}$
by the isomorphism $U_{+}\ni(x,v_{+})\mapsto(x,x^{2-d}v_{+}+x^{1-d})\in U_{-}$,
endowed with the $\mathbb{A}^{1}$-fibration $\xi{}_{d}:W_{d}\to A$
induced by the first projections on each of the open subsets $U_{\pm}$. 
\end{proof}
\begin{cor}
\label{cor:UniquenessReducedA1Fib} For every pair $(\mathbb{F}_{n},B)$
as in Lemma \ref{lem:Pairs} b), the affine surface $\mathbb{F}_{n}\setminus B$
has a unique equivalence class of smooth $\mathbb{A}^{1}$-fibration
$\pi:\mathbb{F}_{n}\setminus B\to\mathbb{A}^{1}$.
\end{cor}

\begin{rem}
\label{rem:TwoLemmas-give-GDIso}The proofs of Lemma \ref{lem:Special-Pencil-Reso}
and Lemma \ref{lem:Uniqueness-Reduced-A1Fib} do not depend on the
Danilov-Gizatullin isomorphism theorem, and, when combined together,
they actually provide a proof of Theorem \ref{thm:DanGiz-Isom}. Indeed,
Lemma \ref{lem:Special-Pencil-Reso} c) asserts in particular that
for every pair $(\mathbb{F}_{n},B)$ there exists a point $q\in B$
such that the $\mathbb{A}^{1}$-fibration $\pi_{q}:\mathbb{F}_{n}\setminus B\to\mathbb{A}^{1}$
associated to the pencil $\mathcal{P}_{q}$ is a smooth morphism.
On the other hand, Lemma \ref{lem:Uniqueness-Reduced-A1Fib} implies
that the isomorphism type of $\pi_{q}:\mathbb{F}_{n}\setminus B\to\mathbb{A}^{1}$
as an $\mathbb{A}^{1}$-fibered surface over $\mathbb{A}^{1}$, hence
in particular as an abstract affine surface, depends only on the integer
$B^{2}$. 
\end{rem}

\subsubsection{Equivalence classes $\mathbb{A}^{1}$-fibrations with an irreducible
component of multiplicity two}

Given a pair $(\mathbb{F}_{n},B)$ as in Lemma \ref{lem:Pairs} b),
denote by $\mathcal{A}_{2}(B^{2})$ the set of equivalence classes
of $\mathbb{A}^{1}$-fibrations $\pi:\mathbb{F}_{n}\setminus B\to\mathbb{A}^{1}$
whose unique degenerate fiber has an irreducible component of multiplicity
two. By Corollary \ref{cor:GizDanSurface-A1FiberType} and Lemma \ref{lem:A1Fib-Mult-Existence},
$\mathcal{A}_{2}(2)=\emptyset$ and $\mathcal{A}_{2}(d)\neq\emptyset$
for every $d\geq3$. The aim of this subsection is to establish the
following more precise description of the sets $\mathcal{A}_{2}(d)$
for $d\geq3$. 
\begin{prop}
\label{thm:MainThm-Mult2}With the notation above, the following hold:

a) The sets $\mathcal{A}_{2}(3)$ and $\mathcal{A}_{2}(4)$ both consist
of a single element, 

b) The sets $\mathcal{A}_{2}(5)$ and $\mathcal{A}_{2}(6)$ both consist
of two elements, 

c) For every $d\geq7$, $\mathcal{A}_{2}(d)$ has cardinality larger
than or equal to that of the field $k$. 
\end{prop}

The proof follows from a combination of several intermediate results
established below. By Lemma \ref{lem:A1-fib-inducedbypencil}, for
a surface $S=\mathbb{F}_{n}\setminus B$ every element of $\mathcal{A}_{2}(B^{2})$
is represented by an $\mathbb{A}^{1}$-fibration $\pi_{q}:S\to\mathbb{A}^{1}$
associated to a pencil $\mathcal{P}_{q}$ on some smooth completion
$(\mathbb{F}_{n'},B')$ of $S$ whose unique singular member is a
divisor of the form $C+2F_{q}$ for some prime element $C$ of the
complete linear system $|B'-2F|$. The unique degenerate fiber of
$\pi_{q}$ then consists of the disjoint union of $C\cap S$ with
multiplicity one and of $F_{q}\cap S$ with multiplicity two. 
\begin{lem}
\label{lem:Reso-Mult(1,2)}Let $(\mathbb{F}_{n},B)$ be pair as in
Lemma \ref{lem:Pairs} b) with $d=B^{2}\geq3$ and such that there
exists a point $q\in B$ for which the singular member of the pencil
$\mathcal{P}_{q}$ is a divisor of the form $C+2F_{q}$ for some prime
element $C$ of the complete linear system $|B-2F|$. Then the total
transform $\sigma_{*}^{-1}B\cup D_{q}\cup\sigma_{*}^{-1}(C)\cup\sigma_{*}^{-1}(F_{q})$
of $B\cup C\cup F_{q}$ in the minimal resolution $\sigma:\tilde{\mathbb{F}}_{n}\to\mathbb{F}_{n}$
of the rational map $\rho_{q}:\mathbb{F}_{n}\dashrightarrow\mathbb{P}^{1}$
defined by $\mathcal{P}_{q}$ is a rational tree of the form \[\xymatrix@R=0.5em @C=0.9em{(\sigma_*^{-1}B,0) \ar@{}[r]|\triangleleft & (E_d,-1)\ar@{}[r]|\triangleleft & (E_{d-1},-2) \ar@{}[r]|\triangleleft & (E_{d-2},-2)\ar@{-}[d] \ar@{}[r]|\triangleleft & \cdots\cdots \ar@{}[r]|\triangleleft & (E_{1},-2) \ar@{-}[r] & (\sigma_*^{-1}F_q,-1) \\ & & & (\sigma_*^{-1}C,-2).}\]  
\end{lem}

\begin{proof}
The assertion is straightforward to verify from the description of
$\sigma:\tilde{\mathbb{F}}_{n}\to\mathbb{F}_{n}$ given in subsection
\ref{subsec:Special-pencils}. 
\end{proof}
\begin{notation}
\label{nota:hatS_q}Let $(\mathbb{F}_{n},B)$ be a pair as in Lemma
\ref{lem:Pairs} b) with $d=B^{2}\geq3$ and such that there exists
a point $q\in B$ for which the singular member of $\mathcal{P}_{q}$
is a divisor of the form $C+2F_{q}$. We denote by $\hat{S}_{q}$
the affine open subset $\mathbb{F}_{n}\setminus(B\cup C)$ of $S=\mathbb{F}_{n}\setminus B$
and we denote by $\hat{\pi}_{q}:\hat{S}_{q}\to\mathbb{A}^{1}$ the
$\mathbb{A}^{1}$-fibration with degenerate fiber $F_{q}\cap\hat{S}_{q}$
of multiplicity two induced by $\pi_{q}:S\to\mathbb{A}^{1}$. With
the notation of Lemma \ref{lem:Reso-Mult(1,2)}, the $\mathbb{P}^{1}$-fibered
surface $\tilde{\rho}_{q}=\rho_{q}\circ\sigma:\tilde{\mathbb{F}}_{n}\to\mathbb{P}^{1}$
is a relatively minimal SNC completion of $\hat{\pi}_{q}:\hat{S}_{q}\to\mathbb{A}^{1}$
with boundary divisor $\hat{D}_{q}=\sigma_{*}^{-1}(B)\cup D_{q}\cup\sigma_{*}^{-1}C=\sigma_{*}^{-1}B\cup\bigcup_{i=1}^{d}E_{i}\cup\sigma_{*}^{-1}C$. 
\end{notation}

Let $(\mathbb{F}_{n},B)$ and $(\mathbb{F}_{n'},B')$ with $B^{2}=(B')^{2}=d\geq3$
be smooth completions of an affine surface $S$ such that for some
points $q\in B$ and $q'\in B'$, the pencils $\mathcal{P}_{q}$ and
$\mathcal{P}_{q}'$ have a singular member of the form $C+2F_{q}$
and $C'+2F_{q'}'$ respectively. Let $\pi_{q}:S\to\mathbb{A}^{1}$,
$\hat{\pi}_{q}:\hat{S}_{q}=S\setminus C\to\mathbb{A}^{1}$, $\pi_{q'}:S\to\mathbb{A}^{1}$
and $\hat{\pi}_{q'}:\hat{S}_{q'}=S\setminus C'\to\mathbb{A}^{1}$
be the associated $\mathbb{A}^{1}$-fibrations. The next lemma reduces
the study of equivalence classes of $\mathbb{A}^{1}$-fibered surfaces
$(S,\pi_{q})$ to those of the simpler ones $(\hat{S}_{q},\hat{\pi}_{q})$. 
\begin{lem}
\label{lem:Mult(1,2)-to-Mult2}The $\mathbb{A}^{1}$-fibered surfaces
$(S,\pi_{q})$ and $(S,\pi_{q'})$ are equivalent if and only if $(\hat{S}_{q},\hat{\pi}_{q})$
and $(\hat{S}_{q'},\hat{\pi}_{q'})$ are equivalent. 
\end{lem}

\begin{proof}
Since $C\cap S$ (resp. $C'\cap S)$ is a reduced fiber of $\pi_{q}$
(resp. $\pi_{q'}$) whereas $F_{q}\cap S$ (resp. $F_{q'}\cap S$)
is a fiber of multiplicity two of it, every equivalence of $\mathbb{A}^{1}$-fibered
surface $\Psi:(S,\pi_{q})\to(S,\pi_{q'})$ maps $C\cap S$ onto $C'\cap S$
and $F_{q}\cap S$ onto $F_{q'}\cap S$, hence restricts to equivalence
between $\hat{\Psi}:(\hat{S}_{q},\hat{\pi}_{q})\to(\hat{S}_{q'},\hat{\pi}_{q'})$.
Now assume conversely that there exists an equivalence of $\mathbb{A}^{1}$-fibered
surfaces $\hat{\Psi}:(\hat{S}_{q},\hat{\pi}_{q})\to(\hat{S}_{q'},\hat{\pi}_{q'})$.
With the notation of Lemma \ref{lem:Reso-Mult(1,2)}, let 
\[
(\tilde{\mathbb{F}}_{n},\hat{D}_{q}=\sigma_{*}^{-1}B\cup D_{q}\cup\sigma_{*}^{-1}C)\quad\textrm{and}\quad(\tilde{\mathbb{F}}_{n'},\hat{D}_{q'}'={\sigma'}_{*}^{-1}B'\cup D_{q'}'\cup{\sigma'}_{*}^{-1}C')
\]
be the relatively minimal SNC completions of $\hat{\pi}_{q}$ and
$\hat{\pi}_{q'}$ respectively. The isomorphism $\hat{\Psi}$ induces
a birational map of $\mathbb{P}^{1}$-fibered surfaces $\hat{\Psi}:(\tilde{\mathbb{F}}_{n},\tilde{\rho}_{q})\dashrightarrow(\tilde{\mathbb{F}}_{n'},\tilde{\rho}_{q'})$.
In particular, $\hat{\Psi}$ maps the section $E_{d}$ of $\tilde{\rho}_{q}$
isomorphically onto the section $E_{d}'$ of $\tilde{\rho}_{q'}$,
which implies that any proper base point of $\hat{\Psi}$ is supported
either on $D_{q}\cup\sigma_{*}^{-1}C$ or on $\sigma_{*}^{-1}B$.
Assume that $\hat{\Psi}$ has a proper base point supported on $D_{q}\cup\sigma_{*}^{-1}C$
and let $\tilde{\mathbb{F}}_{n}\stackrel{\eta}{\leftarrow}Z\stackrel{\eta'}{\to}\tilde{\mathbb{F}}_{n'}$
be the minimal resolution of $\hat{\Psi}$. Since $\hat{\Psi}$ maps
$E_{d}$ isomorphically onto $E_{d}'$ and $F_{q}$ onto $F_{q'}$,
it follows that $\eta'$ contracts $\eta^{-1}(D_{q}\cup\sigma_{*}^{-1}C)$
onto $D_{q'}'\cup{\sigma'}_{*}^{-1}C'$. Since $D_{q}\cup\sigma_{*}^{-1}C$
and $D_{q'}'\cup{\sigma'}_{*}^{-1}C'$ are SNC divisors with the same
number of irreducible components, the minimality assumption implies
that the proper transform in $Z$ of $D_{q}\cup\sigma_{*}^{-1}C$
contains an $\eta'$-exceptional $(-1)$-curve. But this is impossible
since all these curves have self-intersection $\leq-2$ in $Z$. For
the same reason, $\hat{\Psi}^{-1}$ has no proper base point on $D_{q'}'\cup{\sigma'}_{*}^{-1}C'$.
So, $\hat{\Psi}$ is well-defined on an open neighborhood $U$ of
$\hat{D}_{q}\setminus\sigma_{*}^{-1}B$ in $\tilde{\mathbb{F}}_{n}$
and $\hat{\Psi}|_{U}$ is an isomorphism onto an open neighborhood
$U'$ of $\hat{D}_{q'}\setminus{\sigma'}_{*}^{-1}B'$ in $\tilde{\mathbb{F}}_{n'}$.
The geometry of the divisors $\hat{D}_{q}\setminus\sigma_{*}^{-1}B$
and $\hat{D}_{q'}\setminus{\sigma'}_{*}^{-1}B'$ and the fact that
$\hat{\Psi}$ maps $\sigma_{*}^{-1}F_{q}$ onto ${\sigma'}_{*}^{-1}F_{q'}'$
imply that $\hat{\Psi}(\sigma_{*}^{-1}C)={\sigma'}_{*}^{-1}C'$ and
hence, that $\hat{\Psi}$ induces an equivalence of $\mathbb{A}^{1}$-fibered
surfaces $\Psi:(S,\pi_{q})\to(S,\pi_{q'})$. 
\end{proof}
To study equivalence classes of $\mathbb{A}^{1}$-fibered surfaces
$(\hat{S}_{q},\hat{\pi}_{q})$, we now introduce two auxiliary families
of surfaces. 
\begin{notation}
\label{nota:tildeS_l,s}For every integer $\ell\geq1$ and every polynomial
$s\in k[x^{2}]\subset k[x]$ of degree $<\ell$ with $s(0)=1$, denote
by $\tilde{S}_{\ell,s}$ the surface in $\mathbb{A}^{3}=\mathrm{Spec}(k[x,y,z])$
with equation $x^{\ell}z=y^{2}-s^{2}(x)$. The morphism $\tilde{\pi}_{\ell,s}=\mathrm{pr}_{x}:\tilde{S}_{\ell,s}\to\mathbb{A}^{1}$
is a smooth $\mathbb{A}^{1}$-fibration with unique degenerate fiber
$\tilde{\pi}_{\ell,s}^{-1}(0)$ consisting of two irreducible components
$\{x=y\pm1=0\}$. The morphism $\tilde{\pi}_{\ell,s}$ is equivariant
for the actions of the group $\mu_{2}=\{\pm1\}$ given by $(-x,-y,(-1)^{\ell}z)$
on $\tilde{S}_{\ell,s}$ and by $x\mapsto-x$ on $\mathbb{A}^{1}$.
As a scheme over $\mathbb{A}^{1}$, $\tilde{S}_{\ell,s}$ is $\mu_{2}$-equivariantly
isomorphic to the surface $W_{2x^{\ell}s(x)}^{(-1)^{1-\ell}}$ obtained
by gluing two copies 
\[
U_{\pm}=\tilde{S}_{\ell,s}\setminus\{x=y\mp s(x)=0\}=\mathrm{Spec}(k[x][u_{\pm}]),\quad\textrm{where }u_{\pm}=x^{-\ell}(y-s(x))=(y+s(x))^{-1}z
\]
of $\mathbb{A}^{1}\times\mathbb{A}^{1}$ over $(\mathbb{A}^{1}\setminus\{0\})\times\mathbb{A}^{1}$
by the isomorphism $U_{+}\ni(x,u_{+})\mapsto(x,u_{+}+2x^{-\ell}s(x))\in U_{-}$,
endowed with the $\mu_{2}$-action $U_{+}\ni(x,u_{+})\mapsto(-x,(-1)^{1-\ell}u_{+})\in U_{-}$
and with the $\mathbb{A}^{1}$-fibration $\theta_{2x^{\ell}s(x)}:W_{2x^{\ell}s(x)}^{(-1)^{1-\ell}}\to\mathbb{A}^{1}$
induced by the first projections on each of the open subsets $U_{\pm}$. 
\end{notation}

\begin{notation}
\label{nota:S_l,s}The categorical quotient $S_{\ell,s}=\tilde{S}_{\ell,s}/\!/\mu_{2}$
in the category of affine schemes of the fixed point free $\mu_{2}$-action
on $\tilde{S}_{\ell,s}$ is a geometric quotient and the quotient
morphism $\Phi_{\ell,s}:\tilde{S}_{\ell,s}\to S_{\ell,s}$ is a nontrivial
$\mu_{2}$-torsor, in particular, $S_{\ell,s}$ is a smooth affine
surface. The $\mathbb{A}^{1}$-fibration $\tilde{\pi}_{\ell,s}:\tilde{S}_{\ell,s}\to\mathbb{A}^{1}$
descends to an $\mathbb{A}^{1}$-fibration $\pi_{\ell,s}:S_{\ell,s}\to\mathbb{A}^{1}/\!/\mu_{2}=\mathrm{Spec}(k[x^{2}])$
with $\pi_{\ell,s}^{-1}(0)$ as a unique degenerate fiber and we have
a commutative diagram \[\xymatrix{\tilde{S}_{\ell,s} \ar[r]^-{\Phi_{\ell,s}} \ar[d]_{\tilde{\pi}_{\ell,s}} & S_{\ell,s}=\tilde{S}_{\ell,s}/\!/\mu_2 \ar[d]^{\pi_{\ell,s}} \\ \mathbb{A}^1=\mathrm{Spec}(k[x]) \ar[r]^{\phi} & \mathbb{A}^1=\mathrm{Spec}(k[x^2]),}\]
where $\phi:\mathbb{A}^{1}\to\mathbb{A}^{1}/\!/\mu_{2}$ is the quotient
morphism induced by the inclusion $k[x^{2}]\subset k[x]$. Since $\Phi_{\ell,s}^{-1}(\pi_{\ell,s}^{-1}(0))=\tilde{\pi}_{\ell,s}^{-1}(0)$
consists of two component which are exchanged by the $\mu_{2}$-action
and since $\Phi$ is \'etale whereas $\phi$ is totally ramified
of ramification index $2$ over $0$, it follows that $\pi_{\ell,s}^{-1}(0)$
is irreducible, of multiplicity $2$. The Picard group $\mathrm{Pic}(S_{\ell,s})$
is isomorphic to $\mathbb{Z}_{2}$, generated by the class of the
ideal sheaf of the curve $F_{\ell,s}=(\pi_{\ell,s}^{-1}(0))_{\mathrm{red}}$
of $S_{\ell,s}$. 
\end{notation}

\begin{lem}
\label{lem:SmoothMult2-S_l,s}Every smooth affine surface $S$ endowed
with an $\mathbb{A}^{1}$-fibration $\pi:S\to\mathbb{A}^{1}$ whose
unique degenerate fiber is irreducible and of multiplicity two is
equivalent to an $\mathbb{A}^{1}$-fibered surface $\pi_{\ell,s}:S_{\ell,s}\to\mathbb{A}^{1}$
for some $\ell\geq1$. 
\end{lem}

\begin{proof}
We can assume that $\mathbb{A}^{1}=\mathrm{Spec}(k[t])$ and that
$\pi^{-1}(0)$ is the degenerate fiber of $\pi$. Let $\phi:A=\mathrm{Spec}(k[x])\to\mathbb{A}^{1}$
be the ramified $\mu_{2}$-cover $x\mapsto t=x^{2}$ and let $\nu:\tilde{S}\to S\times_{\pi,\mathbb{A}^{1},\phi}A$
be the normalization of $S\times_{\pi,\mathbb{A}^{1},\phi}A$, endowed
with the $\mu_{2}$-action lifting that on the second factor. By \cite[Example 1.6 and Theorem 1.7]{Fie94},
$\tilde{\pi}=\mathrm{pr}_{1}\circ\nu:\tilde{S}\to A$ is $\mu_{2}$-equivariantly
isomorphic as an $A$-scheme to an affine surface $W_{f}^{\varepsilon}$
obtained by gluing two copies $U_{\pm}=\mathrm{Spec}(k[x][u_{\pm}])$
of $A\times\mathbb{A}^{1}$ over $(A\setminus\{0\})\times\mathbb{A}^{1}$
by an isomorphism of the form $U_{+}\ni(x,u_{+})\mapsto(x,u_{+}+f(x))\in U_{-}$
for some $f\in k[x^{-1}]\setminus k$, endowed with a $\mu_{2}$-action
of the form $U_{+}\ni(x,u_{+})\mapsto(-x,\varepsilon u_{+})\in U_{-}$,
where $\varepsilon=1$ or $-1$, viewed as a scheme over $A$ via
the $\mathbb{A}^{1}$-fibration $\theta_{f}:W_{f}^{\varepsilon}\to A$
induced by the first projections on each of the open subsets $U_{\pm}$.
If $\varepsilon=1$, we have $f(-x)=-f(x)$, which implies that the
pole order $\ell=-\mathrm{ord}_{0}f$ of $f$ at $0$ is odd. The
polynomial $\sigma=\frac{1}{2}x^{\ell}f(x)\in k[x^{2}]\setminus x^{2}k[x^{2}]$
can be written in the form $\sigma=\lambda(s(x)+x^{\ell}r(x))$ for
some $s\in k[x^{2}]$ of degree $<\ell$ with $s(0)=1$ and $\lambda\in k^{*}$
and the local isomorphisms
\[
U_{\pm}=\mathrm{Spec}(k[x][u_{\pm}])\to U'_{\pm}=\mathrm{Spec}(k[x][u'_{\pm}]),\,(x,u_{\pm})\to(x,\lambda^{-1}(u_{\pm}\pm r(x)))
\]
glue to a $\mu_{2}$-equivariant isomorphism $W_{f}^{1}\to W_{2x^{-\ell}s(x)}^{1}\cong\tilde{S}_{\ell,s}$
of $A$-schemes. The latter induces in turn an isomorphism $S=\tilde{S}/\!/\mu_{2}=W_{f}^{1}\cong\tilde{S}_{\ell,s}/\!/\mu_{2}=S_{\ell,s}$
of $\mathbb{A}^{1}$-fibered surfaces over $\mathbb{A}^{1}$. If $\varepsilon=-1$,
then $f(-x)=f(x)$, $\ell=-\mathrm{ord}_{0}f$ is even and the same
argument shows that $\pi:S\to\mathbb{A}^{1}$ is isomorphic as a scheme
over $\mathbb{A}^{1}$ to $\pi_{\ell,s}:S_{\ell,s}=W_{2x^{-\ell}s(x)}^{-1}/\!/\mu_{2}\to\mathbb{A}^{1}$
where $s$ is the unique polynomial of degree $<\ell$ with $s(0)=1$
such that $\frac{1}{2}x^{\ell}f(x)=\lambda(s(x)+x^{\ell}r(x))\in k[x^{2}]\setminus x^{2}k[x^{2}]$.
\end{proof}
\begin{lem}
\label{lem:EquivClassS_l,s}Let $(S_{\ell_{i},s_{i}},\pi_{\ell_{i},s_{i}})$,
$i=1,2$ be $\mathbb{A}^{1}$-fibered affine surfaces as in Notation
\ref{nota:S_l,s}. Then the following are equivalent:

a) The $\mathbb{A}^{1}$-fibered surfaces $(S_{\ell_{1},s_{1}},\pi_{\ell_{1},s_{1}})$
and $(S_{\ell_{2},s_{2}},\pi_{\ell_{2},s_{2}})$ are equivalent,

b) The surfaces $S_{\ell_{1},s_{1}}$ and $S_{\ell_{2},s_{2}}$ are
isomorphic,

c) There exists $\lambda\in k^{*}$ such that $s_{2}(\lambda x)=s_{1}(x)$. 
\end{lem}

\begin{proof}
The implication a)$\Rightarrow$b) is clear. Now put $S_{1}=S_{\ell_{1},s_{1}}$,
$S_{2}=S_{\ell_{2},s_{2}}$ and assume that there exists an isomorphism
$\Psi:S_{1}\to S_{2}$. Since $\mathrm{Pic}(S_{2})\cong H_{\textrm{ét}}^{1}(S_{2},\mathcal{O}_{S_{2}}^{*})\cong\mathbb{Z}_{2}$
and $H^{0}(S_{2},\mathcal{O}_{S_{2}}^{*})=k^{*}$, the long exact
sequence of \'etale cohomology associated to the short exact sequence
of sheaves of abelian groups
\[
1\to\mu_{2}\to\mathcal{O}_{S_{2}}^{*}\stackrel{f\mapsto f^{2}}{\longrightarrow}\mathcal{O}_{S_{2}}^{*}\to1
\]
implies that $H_{\mathrm{\acute{e}t}}^{1}(S_{2},\mu_{2})=\mathbb{Z}_{2}$,
generated by the class of the $\mu_{2}$-torsor $\Phi_{\ell,s}:\tilde{S}_{\ell_{2},s_{2}}\to S_{2}$.
Since $\Psi\circ\Phi_{\ell_{1},s_{1}}:\tilde{S}_{\ell_{1},s_{1}}\to S_{2}$
is a nontrivial $\mu_{2}$-torsor, it follows that there exists a
unique $\mu_{2}$-equivariant isomorphism $\tilde{\Psi}:\tilde{S}_{\ell_{1},s_{1}}\to\tilde{S}_{\ell_{2},s_{2}}$
such that $\tilde{\Psi}\circ\Phi_{\ell_{1},s_{1}}=\Phi_{\ell_{2},s_{2}}\circ\Psi$.
We deduce from \cite[Proposition 3.6]{DubPo09} that $\ell_{1}=\ell_{2}$
and that there exists a pair $(\lambda,\mu)\in k^{*}\times k^{*}$
such that $s_{2}^{2}(\lambda x)=\mu^{2}s_{1}^{2}(x)$. Since $s_{1}(0)=s_{2}(0)=1$,
the only possibility is that $\mu=\pm1$ and the implication b)$\Rightarrow$c)
follows. The last implication c)$\Rightarrow$a) follows from the
observation that for $\ell_{1}=\ell_{2}=\ell$ and $s_{2}(\lambda x)=s_{1}(x)$,
the morphism $\tilde{\Psi}:\tilde{S}_{\ell,s_{1}}\to\tilde{S}_{\ell,s_{2}}$
defined by $(x,y,z)\mapsto(\lambda x,y,\lambda^{-\ell}z)$ is a $\mu_{2}$-equivariant
equivalence between the $\mathbb{A}^{1}$-fibered surfaces $(\tilde{S}_{\ell,s_{1}},\tilde{\pi}_{\ell,s_{1}})$
and $(\tilde{S}_{\ell,s_{2}},\tilde{\pi}_{\ell,s_{2}})$ which descends
to an equivalence between $(S_{\ell,s_{1}},\pi_{\ell,s_{1}})$ and
$(S_{\ell,s_{2}},\pi_{\ell,s_{2}})$. 
\end{proof}
\begin{example}
\label{exa:Moduli_Sl,s}For every $\ell\geq5$, put $m=\left\lfloor \tfrac{\ell-3}{2}\right\rfloor $,
$R=k[a_{1},\ldots a_{m}]$ and $\mathfrak{s}(x)=1+x^{2}+\sum_{i=2}^{m}a_{i}x^{2i}\in R[x^{2}]$.
Let $V=\mathrm{Spec}(R)\cong\mathbb{A}^{m}$ and let $\mathfrak{S}_{\ell}$
be the quotient of the closed subscheme $\tilde{\mathcal{\mathfrak{S}}}_{\ell}\subset V\times\mathbb{A}^{3}$
with equation $x^{\ell}z=y^{2}-\mathfrak{s}^{2}(x)$ by the $\mu_{2,V}$-action
$(x,y,z)\mapsto(-x,-y,(-1)^{\ell}z)$. By Lemma \ref{lem:EquivClassS_l,s},
the closed fibers of the smooth morphism $\Theta:\mathfrak{S}_{\ell}\to V$
induced by the $\mu_{2}$-invariant projection $\mathrm{pr}_{1}:\tilde{\mathfrak{S}}_{\ell}\to V$
are pairwise non-isomorphic surfaces of the form $S_{\ell,s}$. 
\end{example}

We now relate the family of surfaces $\hat{\pi}_{q}:\hat{S}_{q}\to\mathbb{A}^{1}$
of Notation \ref{nota:hatS_q} to those $\pi_{\ell,s}:S_{\ell,s}\to\mathbb{A}^{1}$
of Notation \ref{nota:S_l,s}. 
\begin{lem}
\label{lem:S_l,s-relativel--minimalComp}An $\mathbb{A}^{1}$-fibered
affine surface $\pi_{\ell,s}:S_{\ell,s}\to\mathbb{A}^{1}$ admits
a relatively minimal SNC completion $(Y_{\ell,s},D_{\ell,s})$ into
a $\mathbb{P}^{1}$-fibered surface $\bar{\pi}_{\ell,s}:Y_{\ell,s}\to\mathbb{P}^{1}$
such that the union of $D_{\ell,s}$ and of the closure $\overline{F}_{\ell,s}$
of $F_{\ell,s}$ is a rational tree of the form \[\xymatrix@R=0.5em @C=0.8em{(F_\infty,0) \ar@{}[r]|\triangleleft & (H,-1)\ar@{}[r]|\triangleleft & (G_0,-2) \ar@{}[r]|\triangleleft & (G_2,-2)\ar@{-}[d] \ar@{}[r]|\triangleleft & \cdots\cdots\cdots \ar@{}[r]|\triangleleft & (G_{\ell+1},-2) \ar@{-}[r] & (\bar{F}_{\ell,s},-1) \\ & & & (G_1,-2),}\] where
$F_{\infty}$ is the fiber of $\bar{\pi}_{\ell,s}$ over $\mathbb{P}^{1}\setminus\mathbb{A}^{1}$,
$H$ is a section of $\bar{\pi}_{\ell,s}$ and $\bar{\pi}_{\ell,s}^{-1}(0)=\overline{F}_{\ell,s}\cup\bigcup_{i=0}^{\ell+1}G_{i}$.
\end{lem}

\begin{proof}
For every $\ell\geq2$ and every polynomial $s_{\ell}\in k[x^{2}]$
of degree $<\ell$ with $s(0)=1$, write $s_{\ell}=s_{\ell-1}+ax^{\ell-1}$
where $s_{\ell-1}\in k[x^{2}]$ is a polynomial of degree $<\ell-1$
and $a\in k$. The endomorphism $(x,y,z)\mapsto(x,y,xz+2as_{\ell-1}+a^{2}x^{\ell-1})$
of $\mathbb{A}^{3}$ induces a $\mu_{2}$-equivariant birational morphism
$\tilde{\sigma}:\tilde{S}_{\ell,s_{\ell}}\to\tilde{S}_{\ell-1,s_{\ell-1}}$
of $\mathbb{A}^{1}$-fibered surfaces. It descends to a birational
morphism $\sigma:S_{\ell,s_{\ell}}\to S_{\ell-1,s_{\ell-1}}$ of $\mathbb{A}^{1}$-fibered
surfaces restricting to an isomorphism over $\mathbb{A}^{1}\setminus\{0\}$
and contracting $F_{\ell,s_{\ell}}$ onto a point $x_{\ell,s_{\ell}}$
of $F_{\ell-1,s_{\ell-1}}$. This morphism $\sigma$ expresses $S_{\ell,s_{\ell}}$
as the surface obtained from $S_{\ell-1,s_{\ell-1}}$ by blowing-up
the point $x_{\ell,s_{\ell}}$ and then removing the proper transform
of $F_{\ell,s_{\ell}}$. Assume that $(Y_{\ell-1,s_{\ell-1}},D_{\ell-1,s_{\ell-1}})$
is a relatively minimal SNC completion of $\pi_{\ell-1,s_{\ell}}:S_{\ell-1,s_{\ell-1}}\to\mathbb{A}^{1}$
into a $\mathbb{P}^{1}$-fibered surface $\overline{\pi}_{\ell-1,s_{\ell-1}}:Y_{\ell-1,s_{\ell-1}}\to\mathbb{P}^{1}$
which satisfies the claimed properties. Then the pair $(Y_{\ell,s_{\ell}},D_{\ell,s_{\ell}})$,
where $\tau:Y_{\ell,s_{\ell}}\to Y_{\ell-1,s_{\ell-1}}$ is the blow-up
of the point $x_{\ell,s_{\ell}}\in\overline{F}_{\ell-1,s_{\ell-1}}$
and $D_{\ell,s}$ is the proper transform of $D_{\ell-1,s_{\ell-1}}\cup\overline{F}_{\ell-1,s_{\ell-1}}$,
endowed the $\mathbb{P}^{1}$-fibration $\bar{\pi}_{\ell,s_{\ell}}=\overline{\pi}_{\ell-1,s_{\ell-1}}\circ\tau$
is a relatively minimal SNC completion of $\pi_{\ell,s_{\ell}}:S_{\ell,s_{\ell}}\to\mathbb{A}^{1}$
which also satisfies the claimed properties. Now the assertion follows
by induction from Example \ref{exa:ResoPencilConics} and the fact
that $\pi_{1,1}:S_{1,1}=\{xz=y^{2}-1\}/\!/\mu_{2}\to\mathbb{A}^{1}$
is isomorphic to the complement of the smooth conic $Q=\{-xz+y^{2}=0\}$
in $\mathbb{P}_{[x:y,z]}^{2}$, endowed with the $\mathbb{A}^{1}$-fibration
associated to the pencil $\mathcal{P}_{[0:0:1]}$. 
\end{proof}
By combining Lemma \ref{lem:Reso-Mult(1,2)}, Lemma \ref{lem:SmoothMult2-S_l,s}
and Lemma \ref{lem:S_l,s-relativel--minimalComp}, we obtain that
every $\mathbb{A}^{1}$-fibered surface $\hat{\pi}_{q}:\hat{S}_{q}\to\mathbb{A}^{1}$
as in Notation \ref{nota:hatS_q} is equivalent to some surface of
the form $\pi_{d-2,s}:S_{d-2,s}\to\mathbb{A}^{1}$ of Notation \ref{nota:S_l,s},
and, conversely, that every equivalence class of $\mathbb{A}^{1}$-fibered
surface $\pi_{d-2,s}:S_{d-2,s}\to\mathbb{A}^{1}$ is realized by an
$\mathbb{A}^{1}$-fibration $\hat{\pi}_{q}:\hat{S}_{q}\to\mathbb{A}^{1}$
induced by restriction of an $\mathbb{A}^{1}$-fibration on an affine
surface $S=\mathbb{F}_{n}\setminus B$ with $B^{2}=d$. Combining
in turn this result with Lemma \ref{lem:Mult(1,2)-to-Mult2}, we obtain
the following: 
\begin{cor}
\label{cor:EquivClassGeomMult2} Let $(\mathbb{F}_{n},B)$ be a pair
as in Lemma \ref{lem:Pairs} b) with $d=B^{2}\geq3$. Then equivalence
classes of $\mathbb{A}^{1}$-fibered surfaces $((\mathbb{F}_{n}\setminus B)_{q},\pi_{q})$
where $q$ ranges through the set of closed points of the boundaries
$B'$ of smooth completions $(\mathbb{F}_{n'},B')$ of $\mathbb{F}_{n}\setminus B$
such that $\mathcal{P}_{q}$ has a singular member of the form $C+2F_{q}$
are in one-to-one correspondence with equivalence classes of $\mathbb{A}^{1}$-fibered
surfaces $(S_{d-2,s},\pi_{d-2,s})$ of Notation \ref{nota:S_l,s}. 
\end{cor}

Proposition \ref{thm:MainThm-Mult2} is now a straightforward consequence
of Corollary \ref{cor:EquivClassGeomMult2} and of the description
of equivalence classes of $\mathbb{A}^{1}$-fibered surfaces $(S_{d-2,s},\pi_{d-2,s})$
given in Lemma \ref{lem:EquivClassS_l,s}. Namely, for $d=3,4$, the
unique equivalences classes are those of $(S_{1,1},\pi_{1,1})$ and
$(S_{2,1},\pi_{2,1})$ respectively. For $d=5$, the two equivalence
classes are those of the surfaces $(S_{3,1},\pi_{3,1})$ and $(S_{3,x^{2}+1},\pi_{3,x^{2}+1})$.
The case $d=6$ is similar. Finally, if $d\geq7$, then Example \ref{exa:Moduli_Sl,s}
provides a family pairwise non-equivalent $\mathbb{A}^{1}$-fibered
surfaces $(S_{d-2,s},\pi_{d-2,s})$ parametrized by the elements of
$k^{m}$, where $m=\left\lfloor \tfrac{d-5}{2}\right\rfloor \geq1$,
showing in particular that the cardinality of $\mathcal{A}_{2}(d)$
is at least equal to that of $k$. 
\begin{rem}
The ``number of moduli'' $m=\left\lfloor \tfrac{d-5}{2}\right\rfloor \geq1$
for equivalence classes of $\mathbb{A}^{1}$-fibered surfaces $(S_{d-2,s},\pi_{d-2,s})$
with a unique singular fiber of multiplicity two deduced from the
explicit family in Example \ref{exa:Moduli_Sl,s} is the same as that
computed by different techniques in \cite{FKZ11}, as can be seen
by taking $k=2$ in Corollary 6.3.20 of \emph{loc. cit.}. The results
in \cite{FKZ11} apply more generally, in particular, to any smooth
affine $\mathbb{A}^{1}$-fibered surface $S\to\mathbb{A}^{1}$ having
a unique singular fiber, irreducible of arbitrary multiplicity $e\geq2$.
On the other hand, it follows from \cite{Fie94} that similarly as
in the case $e=2$ described above, every such surface can be realized
as a quotient of smooth affine surface $\tilde{S}$ endowed with a
smooth $\mathbb{A}^{1}$-fibration $\tilde{\pi}:\tilde{S}\to\mathbb{A}^{1}$
by a suitable free action of a cyclic group $\mu_{e}$ of $e$-th
roots of unity. This suggests the possibility to construct for every
$e\geq2$ explicit families as in Example \ref{exa:Moduli_Sl,s} over
a base scheme $V$ whose dimension equals the number of moduli computed
in \cite[Corollary 6.3.20]{FKZ11}. 
\end{rem}

\subsection{Proof of Theorem \ref{thm:MainThmA1Fib-Aff}}

In this subsection, we finish the proof of Theorem \ref{thm:MainThmA1Fib-Aff}.
Let $S=\mathbb{F}_{n}\setminus B$ for some pair $(\mathbb{F}_{n},B)$
as in Lemma \ref{lem:Pairs} b) with $d=B^{2}\geq2$. If $d\geq7$,
then, by Proposition \ref{thm:MainThm-Mult2}, $S$ has infinitely
many equivalence classes of $\mathbb{A}^{1}$-fibrations $\pi:S\to\mathbb{A}^{1}$.
It remains to show that for every $d\leq6$, the number of equivalence
classes is finite. For every $d\geq2$ and every $m\in1,\ldots,d-1$,
denote by $\mathcal{A}_{m}(d)$ the set of equivalence classes of
$\mathbb{A}^{1}$-fibrations $\pi:S\to\mathbb{A}^{1}$ whose unique
degenerate fiber has an irreducible component of multiplicity $m$
. The following table summarizes the properties of the sets $\mathcal{A}_{m}(d)$: 

\begin{table}[H]
\begin{centering}
\begin{tabular}{|c|c|c|c|c|c|c|}
\hline 
 & $\sharp\mathcal{A}_{1}(d)$ & $\sharp\mathcal{A}_{2}(d)$ & $\sharp\mathcal{A}_{3}(d)$ & $\sharp\mathcal{A}_{4}(d)$ & $\sharp\mathcal{A}_{5}(d)$ & $\sum\sharp\mathcal{A}_{i}(d)$\tabularnewline
\hline 
\hline 
$d=2$ & $1$ & $0$ & $0$ & $0$ & $0$ & $1$\tabularnewline
\hline 
$d=3$ & $1$ & $1$ & $0$ & $0$ & $0$ & $2$\tabularnewline
\hline 
$d=4$ & $1$ & $1$ & $1$ & $0$ & $0$ & $3$\tabularnewline
\hline 
$d=5$ & $1$ & $2$ & $1$ & $1$ & $0$ & $5$\tabularnewline
\hline 
$d=6$ & $1$ & $2$ & $2$ & $1$ & $1$ & $7$\tabularnewline
\hline 
\end{tabular}
\par\end{centering}
\caption{Numbers of equivalence classes of $\mathbb{A}^{1}$-fibrations}
\end{table}

\vspace{-0.5em}

Indeed, we have $\mathcal{A}_{m}(d)=\emptyset$ if $m\geq d$ by Corollary
\ref{cor:GizDanSurface-A1FiberType}. On the other hand, the cardinal
$\sharp\mathcal{A}_{m}(d)$ of $\mathcal{A}_{m}(d)$ is larger than
or equal to $1$ for every $1\leq m\leq d-1$ by Lemma \ref{lem:A1Fib-Mult-Existence}.
The sets $\mathcal{A}_{d-1}(d)$ and $\mathcal{A}_{1}(d)$ both consist
of a single element by Proposition \ref{cor:UniqueMaximal} and Corollary
\ref{cor:UniquenessReducedA1Fib} respectively. By Proposition \ref{thm:MainThm-Mult2},
we have $\sharp\mathcal{A}_{2}(d)=1$ for $d=3,4$ and $\sharp\mathcal{A}_{2}(d)=2$
for $d=5,6$. These observations settle the cases $d=2$, $3$ and
$4$. In the next paragraphs, we determine the remaining numbers of
equivalence classes of $\mathbb{A}^{1}$-fibrations displayed in the
table. We refer the reader to \cite[Section 4]{DubLam15} for the
details of the reductions to the chosen particular models of pairs
which are used in the argument. 

\vspace{-0.4em}

\subsubsection{The case $d=5$ }

The $\mathbb{A}^{1}$-fibrations on $S$ representing elements of
$\mathcal{A}_{3}(5)$ can only arise form pencils $\mathcal{P}_{q}$
on pairs $(\mathbb{F}_{1},B)$ for which $B\sim C_{0}+3F$ intersects
$C_{0}$ with multiplicity two at a single point $q$. Up to isomorphism,
there is a unique such pair, which is given, under the identification
of $\mathbb{F}_{1}$ with the blow-up $\sigma:\mathbb{F}_{1}\to\mathbb{P}^{2}$
of $\mathbb{P}^{2}$ with homogeneous coordinates $[x:y:z]$ at the
point $p=[0:1:0]$ with exceptional divisor $C_{0}$, by taking for
$B$ the proper transform of the cuspidal cubic $C=\{x^{3}-z^{2}y=0\}$
in $\mathbb{P}^{2}$. The section $B=\sigma_{*}^{-1}C\sim C_{0}+3F$
intersects $C_{0}$ with multiplicity two at the intersection point
$q$ of $C_{0}$ with the proper transform of the tangent line $T_{p}C=\{z=0\}$
to $C$ at $p$ and the pencil $\mathcal{P}_{q}$ is the proper transform
of the pencil generated by $C$ and $3T_{p}C$. The associated $\mathbb{A}^{1}$-fibration
$\pi_{q}:S\to\mathbb{A}^{1}$ has $F_{q}\cap S$ as a component of
multiplicity three in its degenerate fiber. We conclude that $\sharp\mathcal{A}_{3}(5)=1$. 

\vspace{-0.3em}

\subsubsection{The case $d=6$ }

The numbers to be computed are $\sharp\mathcal{A}_{3}(6)$ and $\sharp\mathcal{A}_{4}(6)$.
The possible smooth completions $(\mathbb{F}_{n},B)$ of $S$ are
either of the form $(\mathbb{F}_{0},B)$ where $B\sim C_{0}+3F$ is
a section of $\pi_{0}$, or of the form $(\mathbb{F}_{2},B)$ where
$B\sim C_{0}+4F$ is a section of $\pi_{2}$, or the form $(\mathbb{F}_{4},B)$
where $B\sim C_{0}+5F$ is a section of $\pi_{4}$. 

The $\mathbb{A}^{1}$-fibrations on $S$ representing elements of
$\mathcal{A}_{4}(6)$ can arise only from pairs $(\mathbb{F}_{2},B)$
for which $B\sim C_{0}+4F$ intersects $C_{0}$ with multiplicity
two in a single point. Up to isomorphism, there exists a unique such
pair which is given, after fixing a fiber $F_{\infty}$ of $\pi_{2}$
and an identification $\mathbb{F}_{2}\setminus(C_{0}\cup F_{\infty})\cong\mathbb{A}^{2}=\mathrm{Spec}(k[x,y])$
in such way that $\pi_{2}|_{\mathbb{A}^{2}}=\mathrm{pr}_{x}$ and
that the closures in $\mathbb{F}_{2}$ of the level sets of $y$ are
sections of $\pi_{2}$ linearly equivalent to $C_{0}+2F$, by taking
for $B$ the closure of the curve $\Gamma_{x^{4}}=\{y=x^{4}\}\subset\mathbb{A}^{2}$.
For the point $q=B\cap C_{0}=B\cap F_{\infty}$, the singular member
of the pencil $\mathcal{P}_{q}$ is equal to $C_{0}+4F_{\infty}$.
The unique degenerate fiber of the corresponding $\mathbb{A}^{1}$-fibration
$\pi_{q_{4}}:S\to\mathbb{A}^{1}$ has $F_{\infty}\cap S$ as a component
of multiplicity four and we conclude that $\sharp\mathcal{A}_{4}(6)=1$. 

The $\mathbb{A}^{1}$-fibrations representing elements of $\mathcal{A}_{3}(6)$
can arise only from pairs $(\mathbb{F}_{0},B)$ on $\pi_{0}=\mathrm{pr}_{1}:\mathbb{F}_{0}=\mathbb{P}^{1}\times\mathbb{P}^{1}\to\mathbb{P}^{1}$
for which $B\sim C_{0}+3F$ intersects a fiber of the second projection
with multiplicity three at some point $q$. Up to isomorphisms, there
are exactly two such pairs $(\mathbb{F}_{0},B_{1})$ and $(\mathbb{F}_{0},B_{2})$
which, using bi-homogeneous coordinates $([u_{0}:u_{1}],[v_{0}:v_{1}])$
on $\mathbb{P}^{1}\times\mathbb{P}^{1}$, are given by the curves
$B_{1}=\{u_{1}^{3}v_{0}+u_{0}^{2}(u_{0}+u_{1})v_{1}=0\}$ and $B_{2}=\{u_{1}^{3}v_{0}+u_{0}^{3}v_{1}=0\}$.
The only fiber of $\mathrm{pr}_{2}$ which intersects $B_{1}$ in
a single point is the curve $C_{[1:0]}=\{v_{1}=0\}$ with $q=C_{[1,0]}\cap B_{1}=([1:0],[1:0])$.
This yields an $\mathbb{A}^{1}$-fibration $\pi_{q}:S\to\mathbb{A}^{1}$
which has $F_{q_{1}}\cap S$ as component of multiplicity three in
its degenerate fiber. In contrast, there are two fibers of $\mathrm{pr}_{2}$
which intersects $B_{2}$ in a single point: the curve $C_{[0:1]}=\{v_{0}=0\}$
at the point $q_{0}=([0:1],[0:1])$ and the curve $C_{[1:0]}$ at
the point $q_{\infty}=([1:0],[1:0])$. The $\mathbb{A}^{1}$-fibrations
$\pi_{q_{0}}:S\to\mathbb{A}^{1}$ and $\pi_{q_{\infty}}:S\to\mathbb{A}^{1}$
associated to the pencils $\mathcal{P}_{q_{0}}$ and $\mathcal{P}_{q_{\infty}}$
have $F_{q_{0}}\cap S$ and $F_{q_{\infty}}\cap S$ as components
of multiplicity three of their respective degenerate fibers, and since
the points $q_{0}$ and $q_{\infty}$ belongs to the same orbit of
the action of the group $\mathrm{Aut}(\mathbb{F}_{0},B_{2})\cong\mathbb{G}_{m}\times\mathbb{Z}_{2}$,
these $\mathbb{A}^{1}$-fibrations represents a same element of $\mathcal{A}_{3}(6)$.
The next lemma shows that $\mathcal{A}_{3}(6)$ consists of two elements
and completes the proof.
\begin{lem}
The $\mathbb{A}^{1}$-fibration $\pi_{q}:S\to\mathbb{A}^{1}$ is not
equivalent to $\pi_{q_{0}}:S\to\mathbb{A}^{1}$ (hence not equivalent
to $\pi_{q_{\infty}}:S\to\mathbb{A}^{1}$).
\end{lem}

\begin{proof}
The curve $C_{[1:0]}\cap S\cong\mathbb{A}^{1}$ is a $3$-section
of $\pi_{q_{0}}:S\to\mathbb{A}^{1}$ which intersects the multiple
irreducible component $F_{q_{0}}\cap S$ of the degenerate fiber of
$\pi_{q_{0}}$ transversely in a single point. To verify that $\pi_{q}:S\to\mathbb{A}^{1}$
is not equivalent to $\pi_{q_{0}}:S\to\mathbb{A}^{1}$, it suffices
to show that there is no $3$-section of $\pi_{q}:S\to\mathbb{A}^{1}$
isomorphic to $\mathbb{A}^{1}$ and intersecting $F_{q}$ transversely
in a single point. Suppose on the contrary that such a $3$-section
$D$ exists. Let $\sigma:\tilde{\mathbb{F}}_{0}\to\mathbb{F}_{0}$
be the minimal resolution of the rational map $\rho_{q}:\mathbb{F}_{0}\dashrightarrow\mathbb{P}^{1}$
defined by $\mathcal{P}_{q}$. The closure $\bar{D}$ of $D$ in $\tilde{\mathbb{F}}_{0}$
is a rational $3$-section of the $\mathbb{P}^{1}$-fibration $\tilde{\rho}_{q}=\rho_{q}\circ\sigma$
which intersects the proper transform $\sigma_{*}^{-1}B_{1}$ of $B_{1}$
with multiplicity $3$ in a single point $p$. The total transform
of $B_{1}\cup F_{q}\cup C_{[1:0]}$ in $\tilde{\mathbb{F}}_{0}$ being
rational tree of the form \[\xymatrix@R=0.5em @C=0.6em{(\sigma_*^{-1}B,0) \ar@{}[r]|\triangleleft & (E_6,-1)\ar@{}[r]|\triangleleft & (E_5,-2) \ar@{}[r]|\triangleleft & (E_4,-2)\ar@{}[r]|\triangleleft & (E_3,-2)\ar@{}[r]|\triangleleft \ar@{-}[d] &  (E_2,-2) \ar@{}[r]|\triangleleft &  (E_1,-2) \ar@{-}[r] & (\sigma_*^{-1}F_q,-1) \\ & & & & (\sigma_*^{-1}C_{[1:0]},-3),}\] 
there exists a unique birational morphism of $\mathbb{P}^{1}$-fibered
surface $\tau:\tilde{\mathbb{F}}_{0}\to\mathbb{F}_{1}$ which contracts
$\sigma_{*}^{-1}F_{q}\cup\sigma_{*}^{-1}C_{[1:0]}\cup\bigcup_{i=1}^{4}E_{i}$
onto a point $s\in\tau(E_{5})\setminus\tau(E_{6})$. The curve $\tau(\bar{D})$
is a $3$-section of $\pi_{1}:\mathbb{F}_{1}\to\mathbb{P}^{1}$ which
has a cusp of multiplicity $2$ at $s$ and intersects $\tau(E_{5})$
with multiplicity $3$ at $s$. Let $C$ be the image of $\tau(\bar{D})$
by the contraction $\alpha:\mathbb{F}_{1}\to\mathbb{P}^{2}$ of $\tau(E_{6})$
to a point $p'$. Assume that $m=\tau(\bar{D})\cdot\tau(E_{6})\geq1$.
Then $C$ is a curve of degree $m+3$ which intersects the line $\alpha(\tau(\sigma_{*}^{-1}B_{1}))$
with multiplicity $m+3$ at $p'$ and the line $\alpha(\tau(E_{5}))$
with multiplicity $m$ at $p'$ and multiplicity $3$ at $\alpha(s)$.
Choosing homogeneous coordinates $[x:y:z]$ on $\mathbb{P}^{2}$ so
that $\alpha(\tau(\sigma_{*}^{-1}B_{1}))=\{z=0\}$, $\alpha(\tau(E_{5}))=\{x=0\}$
and $\alpha(s)=[0:0:1]$, the curve $C$ is thus given by an equation
of the form $\lambda x^{m+3}-\mu y^{3}z^{m}=0$ for some $\lambda,\mu\in k^{*}$.
But this is impossible since on the other hand $C=\alpha(\tau(\bar{D}))$
has multiplicity $2$ at $\alpha(s)$. So $m=0$ and hence, $C$ is
a cubic with a cusp at $\alpha(s)$ and intersecting $\tau(\sigma_{*}^{-1}B)$
with multiplicity $3$ at a point other than $p'$. It follows that
$\sigma(\bar{D})$ is a smooth rational curve which intersects $F_{q}$
transversely at unique point of $F_{q}\setminus\{q\}$ and $B_{1}$
at a unique point of $B_{1}\setminus\{q\}$, with multiplicity $3$.
Thus, $\sigma(\bar{D})$ is a fiber of $\mathrm{pr}_{2}:\mathbb{P}^{1}\times\mathbb{P}^{1}\to\mathbb{P}^{1}$
which intersects $B_{1}$ with multiplicity $3$ at a point other
than $q$, which is impossible. 
\end{proof}
\begin{acknowledgement*}
The author is very grateful to Mikhail Zaidenberg and Takashi Kishimito
for their helpful comments and valuable suggestions which helped to
improve the completeness and the readability of the article.
\end{acknowledgement*}
\bibliographystyle{amsplain} 

\providecommand{\bysame}{\leavevmode\hbox to3em{\hrulefill}\thinspace}
\providecommand{\MR}{\relax\ifhmode\unskip\space\fi MR }
\providecommand{\MRhref}[2]{%
  \href{http://www.ams.org/mathscinet-getitem?mr=#1}{#2}
}
\providecommand{\href}[2]{#2}
\begin{thebibliography}{}

\end{thebibliography}


\begin{thebibliography}{99}

\bibitem{Be83} J. Bertin, \emph {Pinceaux de droites et automorphismes des surfaces affines}, J. Reine Angew. Math. 341 (1983), 32-53. 

\bibitem{BlanDub11} J. Blanc and A. Dubouloz, \emph{Automorphisms of $\mathbb{A}^{1}$-fibered affine surfaces}, Trans. Amer. Math. Soc. 363 (2011), no. 11, 5887-5924. 

\bibitem{BlanDub15} J. Blanc and A. Dubouloz \emph{Affine surfaces with a huge group of automorphisms}, Int. Math. Res. Not. IMRN 2015, no. 2, 422-459.

\bibitem{BlavS19} J. Blanc  and I. van Santen, \emph{Embeddings of affine spaces into quadrics}, Trans. Amer. Math. Soc. 371 (2019), no. 12, 8429-8465.

\bibitem{CNR07} P. Cassou-Nogu\`es and P. Russell, \emph{Birational morphisms $\mathbb{C}^2\to \mathbb{C}^2$ and affine ruled surfaces}, Affine algebraic geometry, 57-105, Osaka Univ. Press, Osaka, 2007.

\bibitem{GiDa75} V. I. Danilov, M. H. Gizatullin, \emph{Automorphisms of affine surfaces, I}, Izv. Akad. Nauk SSSR Ser. Mat. 39 (1975), no. 3, 523-565.

\bibitem{GiDa77} V. I. Danilov, M. H. Gizatullin, \emph{Automorphisms of affine surfaces, II}, Izv. Akad. Nauk SSSR Ser. Mat. 41 (1977), no. 1, 54-103.

\bibitem{HartBook} R. Hartshorne, \emph{Algebraic geometry}, Graduate Texts in Mathematics, No. 52. Springer-Verlag, New York-Heidelberg, 1977. 

\bibitem{DubThesis} A. Dubouloz, \emph{Sur une classe de sch\'emas avec actions de fibr\'es en droites},  PhD thesis, Grenoble, 2004.
%\url{https://tel.archives-ouvertes.fr/tel-00007733}.

\bibitem{Dub05} A. Dubouloz, \emph{Completions of normal affine surfaces with a trivial Makar-Limanov invariant}, Michigan Math. J. 52 (2004), no. 2, 289-308.

\bibitem{DubAnnalen15} A. Dubouloz, \emph{Complements of hyperplane sub-bundles in projective spaces bundles over $\mathbb{P}^1$}. Math. Ann. 361 (2015), no. 1-2, 259-273.

\bibitem{DubFin14}  A. Dubouloz, D.R. Finston,  \emph{On exotic affine 3-spheres}, J. Algebraic Geom. 23 (2014), no. 3, 445-469.

\bibitem{DubLam15} A. Dubouloz and S. Lamy, \emph{Automorphisms of open surfaces with irreducible boundary}, Osaka J. Math. 52 (2015), no. 3, 747-791.

\bibitem{DubPo09} A. Dubouloz and P.-M. Poloni, \emph{On a class of Danielewski surfaces in affine 3-space}, J. Algebra 321 (2009), no. 7, 1797-1812. 

\bibitem{Fie94} K.-H. Fieseler, \emph{On complex affine surfaces with $\mathbb{C}_+$-action}, Comment. Math. Helv. 69 (1994), no. 1, 5-27. 

\bibitem{FKZ07} H. Flenner, S. Kaliman and M. Zaidenberg, \emph{Birational transformations of weighted graphs}, Affine algebraic geometry, 107--147, Osaka Univ. Press, Osaka, 2007. 

\bibitem{FKZ07-2} H. Flenner, S. Kaliman and M Zaidenberg, \emph{Completions of $\mathbb{C}^*$-surfaces},  Affine algebraic geometry, 149-201, Osaka Univ. Press, Osaka, 2007.

\bibitem{FKZ09} H. Flenner, S. Kaliman and M. Zaidenberg, \emph{On the Danilov-Gizatullin isomorphism theorem}, Enseign. Math. (2) 55 (2009), no. 3-4, 275-283.

\bibitem{FKZ11} H. Flenner, S. Kaliman and M. Zaidenberg, \emph{Smooth affine surfaces with non-unique $\mathbb{C}^*$-actions}, J. Algebraic Geom. 20 (2011), no. 2, 329-398.

\bibitem{Giz70} M. H. Gizatullin, \emph{On affine surfaces that can be completed by a nonsingular rational curve}, Math. USSR Izv. 4 (1970), 787-810.

\bibitem{Giz71} M. H. Gizatullin, \emph{Quasihomogeneous affine surfaces}, Izv. Akad. Nauk SSSR Ser. Mat. 35 (1971), 1047-1071.

\bibitem{Goo69} J. E. Goodman, \emph{Affine open subsets of algebraic varieties and ample divisors}, Ann. of Math. (2) 89 (1969), 160-183. 

\bibitem{Kov15} S. Kovalenko, \emph{Transitivity of automorphism groups of Gizatullin surfaces}, Int. Math. Res. Not. IMRN 2015, no. 21, 11433-11484.

\bibitem{KPZ17} S. Kovalenko, A. Perepechko and M. Zaidenberg, \emph{On automorphism groups of affine surfaces},  Algebraic varieties and automorphism groups, 207-286, Adv. Stud. Pure Math., 75, Math. Soc. Japan, Tokyo, 2017.

\bibitem{MiyBook} M. Miyanishi, \emph{Open {A}lgebraic {S}urfaces}, CRM Monogr. Ser., 12, Amer. Math. Soc., Providence, RI, 2001.


\end{thebibliography}

\end{document}